\documentclass{article}
\usepackage{amsmath,amsthm,amssymb}
\usepackage{graphicx}
\usepackage{enumerate}
\theoremstyle{definition}
\newtheorem{definition}{Definition}[section]

\theoremstyle{plain}
\newtheorem{theorem}[definition]{Theorem}
\newtheorem{lemma}[definition]{Lemma}
\newtheorem{proposition}[definition]{Proposition}

\newtheorem{remark}[definition]{Remark}
\date{\today}
\begin{document}

\title {On the dynamics of skew tent maps}
\author{
{\bf\large Kaijen Cheng}\vspace{1mm}\\
{\it College of Mathematics and Computer Science,}\\
{\it Quanzhou Normal University}\\
{\it Quanzhou 362000, Fujian, P.R.China}\\
{\it\small e-mail: kaijen@cycu.org.tw }\vspace{1mm}\\
{\bf\large Kenneth J. Palmer}\vspace{1mm}\\
{\it Department of Mathematics,}\\
{\it National Taiwan University}\\
{\it Taipei - Taiwan}\\
{\it\small e-mail: palmer@math.ntu.edu.tw}\vspace{1mm}\\
}

\maketitle

\begin{abstract}
In this paper we give an elementary treatment of the dynamics of skew tent maps.
We divide the two-parameter space into six regions.
Two of these regions are further subdivided into infinitely many
regions. All of the regions are given explicitly.
We find the attractor in each subregion,  
determine whether the attractor is a periodic orbit or is chaotic,
and also determine the asymptotic fate of every point.
We find that when the attractor is chaotic, it is either
a single interval or the disjoint union of a finite number
of intervals; when it is a periodic orbit, all periods
are possible. Sometimes, besides the attractor, there exists 
an invariant chaotic Cantor set.
\end{abstract}

\newpage

\section{Introduction}

In this paper we give an elementary treatment of the dynamics of skew tent maps
following on from \cite{CP} where a special case was studied. 
Most of the results given here were already proved in \cite{ITN1, ITN}
but using non-elementary methods.
The papers \cite{II, MV, MV1} also concern skew tent maps
but confine themselves to the discussion of sophisticated concepts such as kneading sequences,
entropy and renormalization. Here our modest aim is to find the attractor,
determine whether the dynamics on the attractor is chaotic or not 
(in the sense of Devaney \cite{D}) and
determine the asymptotic fate of all points and to do these things using
elementary methods and with complete proofs. Two other elementary approaches have been given by
Bassein \cite{B} and Lindstr\" om and Thunberg \cite{LT}.  Also recently a
description of most of the results, but without many of the proofs, has been given by 
Sushko et al. in \cite{SAG}.
\medskip

We divide the parameter space into six regions. Bassein \cite{B}
considers only four of these regions. 
Lindstr\" om and Thunberg \cite{LT} also omit a couple of these regions.
Here we study all six regions and also obtain more 
detail about the dynamics than Bassein \cite{B} or Lindstr\" om and
Thunberg \cite{LT} in the two regions (Sections 4 and 5),
where the more complicated behaviour occurs. 
In fact, the main results here are in Sections 4 and 5 which form the bulk of the paper.
As in Sushko et al. \cite{SAG}, we find the attractor in each subregion but
we also determine the asymptotic fate of every point.
In Section 4 we show that apart from the points whose orbits go to infinity,
all other points except those which are preimages of a finite set of unstable 
periodic orbits go to the attractor, which is a disjoint union of a finite number 
of closed intervals (a so-called band attractor) on which the dynamics is chaotic. 
In Section 5 we show that apart from the points whose orbits go to infinity,
all other points, except a chaotic Cantor set and its preimages or a certain periodic orbit and its preimages, go to the attractor,
which can be a periodic orbit or a band attractor on which the dynamics is chaotic. 
In remarks 3.2, 3.4, 3.6, 3.8, 4.5 and 5.12, we give more information about relation between our results 
and those in \cite{B,ITN1,ITN,LT,SAG}.  
\medskip

More detail about the results in the paper are given at the end of the next section
after we have introduced the requisite notation.

\section{Preliminaries}

A general (continuous) tent map can be defined as follows: we take 2 non-horizontal straight lines,
not parallel, which intersect at a point $(x_0,y_0)$, one with slope $r$ and the 
other with slope $-k$. Then we define our tent map to be
\begin{equation}\label{0} f(x)=\begin{cases} s+rx & (x\le x_0)\\ t-kx & (x\ge x_0),\end{cases}
\end{equation}
where $s+rx_0=t-kx_0=y_0$. 
If $rk<0$, the map is a homeomorphism and the dynamics is trivial. 
Also if $r<0$, $k<0$, the map is conjugate via the transformation $h(x)=-x$ to
\[(h^{-1}fh)(x)=\begin{cases} -t-kx & (x\le -x_0)\\ -s+rx & (x\ge -x_0).\end{cases}\]
So we may assume that $r>0$, $k>0$ in \eqref{0}.
\medskip

Then if we define $H(x)=x+x_0$, we find that
\[ g(x)=(H^{-1}fH)(x)
=\begin{cases} \gamma+rx & (x\le 0)\\ \gamma-kx & (x\ge 0),\end{cases}\]
where $\gamma=y_0-x_0$. 
Now if $\gamma\le 0$,
we see that $g(x)\le 0$ for all $x$. So we can restrict to $x\le 0$
where $g$ is strictly increasing so that the dynamics is trivial.
\medskip

So the
only interesting case is $\gamma>0$. Then if we take $\ell(x)=\gamma x$,
we find that
\[ (\ell^{-1}g\ell)(x)
=\begin{cases} 1+rx & (x\le 0)\\ 1-kx & (x\ge 0),\end{cases}.\]
So, without loss of generality, we can consider maps
\begin{equation}\label{1} f(x)=\begin{cases} 1+rx & (x\le 0)\\ 1-kx& (x\ge 0).\end{cases} \end{equation}
where $r>0$, $k>0$. Note that $f(x)\le 1$ for all $x$.

\begin{remark}
This is essentially the same map studied 
in (3) in \cite{SAG}. Their $a_{{\cal L}}=r$ and $a_{{\cal R}}=-k$.
\end{remark}

First we give some preliminary results about the map in \eqref{1}.

\begin{lemma}\label{lem1} Let $f$ be as in \eqref{1}. 
When $r>1$, set $I=[\alpha,\beta]$, where
\[ \alpha=-\frac{1}{r-1},\quad \beta=\frac{r}{k(r-1)}\]
and when $0<r\le 1$, set $I=(-\infty,\infty)$. Then
\medskip

{\rm (i)} when $r>1$, $f(\alpha)=\alpha$, $f(\beta)=\alpha$ and
$x\notin I\Longrightarrow f^{n}(x)\to -\infty\quad{\rm as}\quad n\to\infty$;
\medskip

{\rm (ii)} for all $x\in{\rm int}(I)$, there exists $n\ge 0$ such that 
$f^{n}(x)\in [1-k,1]$;
\medskip

{\rm (iii)} when $k\le 1$, or $k>1$ and $r\le k/(k-1)$,
then $[1-k,1]\subset I$, $f(I)\subset I$ and $f([1-k,1])\subset [1-k,1]$;
\medskip

{\rm (iv)} when $k>1$ and $r\le k/(k-1)$,
if we define $h:[0,1]\to [1-k,1]$ by $h(x)=1-k+kx$, 
then $h$ is bijective and $g=h^{-1}\circ f\circ h$ maps $[0,1]$ onto itself and is given by 
\begin{equation}\label{2} g(x)=\begin{cases} b+rx & (0\le x\le a)\\ 
                      k(1-x) & (a\le x\le 1),\end{cases}\end{equation}
where $a=1-1/k$, $b=1-ra$.
\end{lemma}

\begin{proof}
(i) Clearly $f(\alpha)=\alpha$, $f(\beta)=\alpha$. If $x<\alpha$,
\[ f(x)-\alpha=r(x-\alpha)<0.\]
It follows that for $n\ge 0$, $f^n(x)<\alpha$ and $f^{n}(x)-\alpha=r^{n}(x-\alpha)$.
Hence $f^{n}(x)-\alpha\to -\infty$ as $n\to\infty$.
If $x>\beta$, then 
\[ f(x)=1-kx<1-k\beta=\alpha\]
and it follows from the previous part that $f^{n}(x)-\alpha\to -\infty$ as $n\to\infty$.
\medskip

(ii) If $\alpha<x<\beta$, then $\alpha<f(x)\le 1$. So we may assume
$\alpha<x\le 1$ where $\alpha=-\infty$ if $r\le 1$.
Then we show there exists $n\ge 0$ such that $f^{n}(x)\in [0,1]$.
Suppose $\alpha< f^{n}(x)<0$ for $n\ge 0$.
If $r>1$, by induction on $n$, it follows that $f^{n}(x)-\alpha=r^{n}(x-\alpha)$
which $\to \infty$ as $n\to\infty$. If $r=1$, $f^{n}(x)=n+x$
which again $\to \infty$ as $n\to\infty$. On the other hand, if $r<1$, 
\[ f^n(x)= 1+r+\cdots+r^{n-1}+r^nx\to 1/(1-r)>0 \]
as $n\to\infty$. So in all cases, there exists $n\ge 0$ such that
$f^{n}(x)\ge 0$ or $f^{n}(x)\le\alpha$. Let $n$ be the first such $n$.
If $n=0$, then $x\ge 0$ since $\alpha<x\le 1$.
If $n\ge 1$, then $\alpha<f^{n-1}(x)<0$
implies that $\alpha=f(\alpha) < f^{n}(x)<1$. Hence $f^{n}(x)\in [0,1]$.
Then $f^{n+1}(x)=1-kf^n(x)\in [1-k,1]$.
\medskip

(iii) First we show $[1-k,1]\subset I$. We need only consider $r>1$.
Then if $k\le 1$ or $k>1$, $r\le k/(k-1)$, it is clear that $\alpha\le 1-k$
and $\beta\ge 1$. It follows that $[1-k,1]\subset I$.
\medskip

Next we show $f(I)\subset I$. If $r\le 1$, this is trivial. If $r>1$, 
\[ \alpha\le x\le 0\Longrightarrow 1+r\alpha\le 1+rx\le 1 \Longrightarrow \alpha\le f(x)\le 1.\]
\[ 0\le x\le \beta\Longrightarrow 1
\ge 1-kx\ge 1-k\beta\Longrightarrow 1\ge f(x)\ge \alpha.\]
However $k\le 1$, or $k>1$, $r\le k/(k-1)$ implies $\beta \ge 1$. Hence $f(I)\subset I$.
\medskip

Next we show $f([1-k,1])\subset [1-k,1]$. When $k>1$,
\[ f([1-k,1])=[1+r(1-k),1]\cup [1-k,1]\subset [1-k,1]\]
if $1+r(1-k)\ge 1-k$, that is, $r\le k/(k-1)$. 
When $0<k\le 1$,
\[ f([1-k,1])=[1-k,1-k(1-k)]\subset [1-k,1].\]

(iv) This is simple algebra.
\end{proof}

\begin{remark}
	Note that $g$ in \eqref{2} is the map studied in \cite{B} and \cite{ITN}.
These authors restrict themselves to the parameter range $k>1$, $0<r<k/(k-1)$,
	thereby excluding the cases in Propositions \ref{prop1} and \ref{prop4} below.
\end{remark}

Now we summarize the results in the rest of the paper.
In Section 3, we study the parameter ranges $k<1$, where there is a stable fixed
point; $k>1$, $r<1/k$ where there is a stable $2-$cycle;
$k>1$, $\max\{1,k/(k^2-1)\}<r<k/(k-1)$ or $k/(k^2-1)<r<1/(k-1)$,
where the map is chaotic on the whole of $[1-k,1]$, and
$k>1$, $r>k/(k-1)$, where the orbits of all points go to infinity
except those lying on a chaotic Cantor set.
\medskip

In Section 4, we study the parameter range $k>1$, $1/k<r<k/(k^2-1)$.
This is divided into subranges (which are explicitly described as in
\cite{SAG}), where the attractor consists of 
a finite number of disjoint closed intervals and, apart from the points whose orbits go to infinity, all other points except a finite set of unstable 
periodic orbits and their preimages, go to the attractor on which the dynamics is chaotic.
\medskip

In Section 5, we study the parameter range $k>1$, $1/(k-1)<r<1$.
Corresponding to each integer $m\ge 2$, there is a subrange.
Inside each subrange there are 4 further subranges, each with a different kind of attractor: (i) a stable periodic
orbit with period $m+1$, (ii) the union of $m+1$ disjoint closed intervals on which
the dynamics is chaotic, (iii) the union of $2(m+1)$ disjoint closed intervals on which
the dynamics is chaotic or (iv) $[1-k,1]$ itself is chaotic. As in \cite{SAG}, the different
subranges are described explicitly. Moreover, we show that in the cases (i)-(ii), 
the orbits of all points, except those which lie on a chaotic Cantor set or are preimages of this set, go to the attractor,
in case (iii) the orbits of all points, except those which lie on a chaotic Cantor set or on a certain periodic orbit or are preimages of
the set or the periodic orbit, go to the attractor,
whereas in case (iv) the orbits of all points go to the attractor.
\medskip

\begin{figure}
\centering
\includegraphics[width=0.80\textwidth]{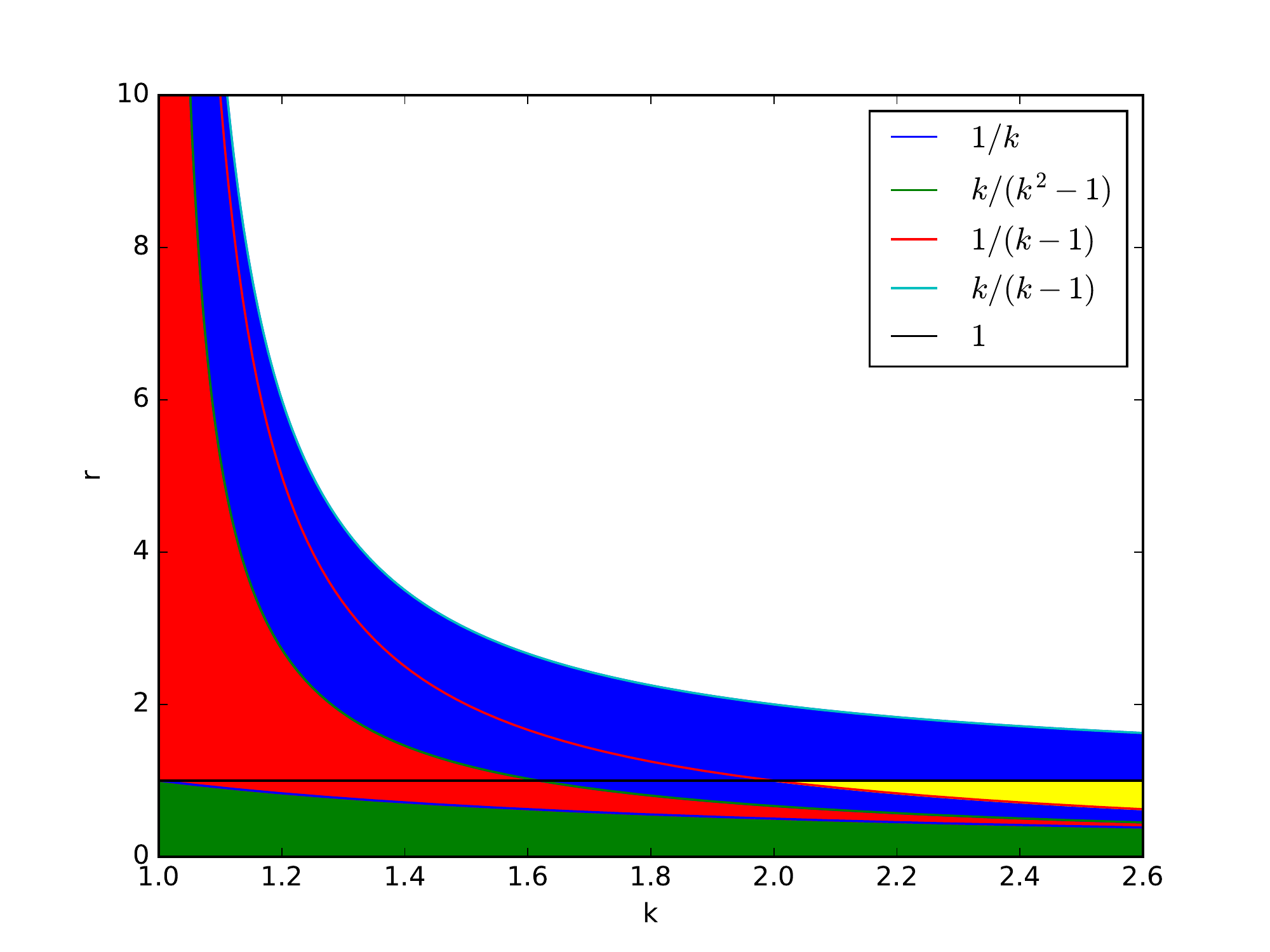}
\caption{Regions in $(k,r)-$parameter space: attracting period 2 cycle (green), 
chaos on $[1-k,1]$ (blue), chaotic band attractors (red), 
attracting periodic orbits and chaotic band attractors (yellow). }
\end{figure}

\section{The simple cases
	}

\subsection{The attractor is a fixed point}

\begin{proposition}\label{prop1}
Let $f$ be as in \eqref{1} and $I$ as in Lemma \ref{lem1}. When $k<1$ and $x$ is in int$(I)$, then $f^n(x)$ converges to 
the fixed point $x^*=\frac{1}{k+1}$. 
\end{proposition}

\begin{proof} From Lemma \ref{lem1} (ii), (iii), we can assume
$f^{n}(x)\in [1-k,1]$ for all $n\ge 0$. Then  
$\vert f(x)-x^*\vert=|1-kx-(1-kx^*)|=k\vert x-x^*\vert$ and 
by induction $\vert f^{n}(x)-x^*\vert=k^{n}\vert x-x^*\vert$ 
for $n\geq 0$. 
So $f^{n}(x)\rightarrow x^*$ as $n\rightarrow \infty.$ 
The proposition follows.
\end{proof}

\begin{remark} This region is not considered in \cite{B} or \cite{ITN}.
The same result is shown in \cite{SAG}; 
see the first two bullet points on page 587 and the
first part of Section 3.1. See also Theorem 4.1 (c), (d) in \cite{LT}.
\end{remark}

\subsection{The attractor is a stable period 2 orbit}

\begin{proposition}\label{prop2}
Let $f$ be as in \eqref{1}. Suppose $k>1$, $r<1/k$. Then $f$ has a stable $2-$cycle 
and all orbits are attracted to this 
$2-$cycle except $x^*=1/(k+1)$ and its preimages. 
\end{proposition}
 
\begin{proof}
First from Lemma \ref{lem1}, since $1/k<k/(k-1)$ when $k>1$,
the interval $[1-k,1]$ is invariant under $f$ and for all real $x$,
there exist $n\ge 0$ such that $f^n(x)\in [1-k,1]$.
So we may restrict $f$ to $[1-k,1]$
and study $g=h^{-1}fh:[0,1]\to [0,1]$ as defined in \eqref{2}.
Then 
\[ b-(1-a+a^2)=\frac{(k-1)(1-rk)}{k^2}>0\]
and it follows from the discussion
in Section 2 of \cite{B} that $g$ has an attracting $2-$cycle and 
all orbits in $[0,1]$ are attracted to this 
$2-$cycle except the fixed point $c=1/(2-a)$ and its preimages. 
So $f=hgh^{-1}$ has an attracting $2-$cycle and all orbits in $[1-k,1]$
are attracted to this $2-$cycle except $x^{*}=h(c)=\frac{1}{k+1}$ and its preimages.
\end{proof}

\begin{remark} Our proof comes from \cite{B}.
The same result is shown in \cite{SAG}; 
see page 590--591. See also Theorem 4.1 (f) in \cite{LT}.
This case is also mentioned in (1) on the first page of \cite{ITN}.
\end{remark}

\subsection{The attractor is a chaotic interval $[1-k,1]$}

\begin{proposition}\label{prop3}
Let $f$ be as in \eqref{1} and $I$ as in Lemma \ref{lem1}. 
Suppose $k>1$ and either $\max\{1,k/(k^2-1)\}<r<k/(k-1)$ or $k/(k^2-1)<r<1/(k-1)$.
Then $[1-k,1]$ is a chaotic attractor for $f$ in {\rm int}$(I)$.
\end{proposition}

\begin{proof} 
In Lemma \ref{lem1}, we saw the interval $[1-k,1]$ is invariant under $f$.
Moreover for all $x\in$ int$(I)$, there exists $n\ge 0$ such that
$f^n(x)\in [1-k,1]$. So we may restrict $f$ to $[1-k,1]$
and study $g=h^{-1}fh:[0,1]\to [0,1]$ as defined in \eqref{2}.

We see that
\[ b-a=\frac{(k-1)(1/(k-1)-r)}{k}>0 \]
when $r<1/(k-1)$, and when $r>1$
\[ \frac{1-b}{a}=r>1\]
so that $b<1-a$. Also
\[ b-c=b-\frac{1}{2-a}=-\frac{rk^2-r-k}{k(k+1)}<0.\]
Then it follows from Propositions 2 and 4 
in \cite{B} that $g$ is chaotic on $[0,1]$ and hence $f$ on $[1-k,1]$.
\end{proof}

\begin{remark} Our proof comes from \cite{B}.
Similar results are shown in \cite{SAG},
although we did not find them easily. They are included
in the results referring to ${\cal A}_{1}$. The first region here is called
$D^*$ in \cite{ITN} and the second region is their $D_1$.
These two regions are not separately considered in \cite{LT}. Also there
are no results about chaos in \cite{LT}, only results concerning
positive Lyapunov exponent.

\end{remark}

\subsection{All points escape except those on a chaotic Cantor set}

\begin{proposition}\label{prop4}
Let $f$ be as in \eqref{1} and $\alpha$, $\beta$ as in Lemma \ref{lem1}. 
Suppose $k>1$, $r>k/(k-1)$.
Then if $x$ is real, either $f^n(x)\to-\infty$ as $n\to\infty$
or $x\in \Lambda\subset [\alpha,\beta]$, where 
\[\Lambda=\{x\in [\alpha,\beta]: f^n(x)\in [\alpha,\beta]\;{\rm for}\; n\ge 0\}\]
is an invariant Cantor set on which $f$ is chaotic.
\end{proposition}

\begin{proof}
If we define $h:(-\infty,\infty)\to(-\infty,\infty)$ by  
\[ h(x)=\frac{r+k}{k(r-1)}x-\frac{1}{r-1},\]
then $h([0,1])=[\alpha,\beta]$ and $g=h^{-1}\circ f\circ h$ is given by
\[ g(x)=\begin{cases} rx & (x\le a)\\ k(1-x) & (x\ge a),\end{cases}\]
where $a=k/(r+k)$ satisfies $g(a)=ra>1$. Then we define
\[ I_0=[0,1/r],\quad I_1=[1-1/k,1]\]
and
\[ \tilde\Lambda=\{x\in [0,1]: g^{n}(x)\in [0,1]\;{\rm for}\; n\ge 0\}.\]
As for the logistic map $\mu x(1-x)$ for $\mu>4$ (see pages 34--38 in \cite{D} or pages
98--101 and 131 in \cite{E}), we show $\tilde\Lambda$ is an invariant Cantor set on which 
$g$ is chaotic and such that if $x\notin \Lambda$, $g^n(x)\to -\infty$ as $n\to\infty$.
The proposition follows with $\Lambda=h(\tilde\Lambda)$.
\end{proof}

\begin{remark} 
Compare equation (14) on page 592 in \cite{SAG}. This region is not considered in \cite{B} or \cite{ITN}.
In Theorem 4.1 (j), it is only shown that almost all orbits tend to $-\infty$. The existence
of the Cantor set is not shown.
\end{remark}

\section{Chaotic band attractors}
 
Now when $k>1$, $k/(k^2-1)<r$ (equivalently $rk^2-k-r>0$) and $r<1/(k-1)$, we see
from Proposition \ref{prop3} that $f$ is chaotic on $[1-k,1]$.
In this section we suppose that 
\begin{equation}\label{s0eq} (k,r)\in S_{1}=\{(k,r): k>1, 1/k<r\le k/(k^{2}-1)\}.
\end{equation}

Now we state the main theorem.

\begin{theorem}\label{prop5} Let $f$ be as in \eqref{1} with $(k,r)\in S_{1}$ and
let $I$ be as in Lemma \ref{lem1}. Then there exists a decreasing sequence
$\{S_{p}\}^{\infty}_{p=1}$ of subsets of $S_{1}$
such that when  $p\ge 1$ and $(k,r)\in {\rm int}(S_p)$, 
\[  \tilde\Lambda(k,r,p)=\bigcup_{i=0}^{2^{p}-1}J_{p,i}=\bigcup_{i=0}^{2^{p}-1}f^{i}([f^{2^{p}}(1),1]).\]
is invariant
under $f$, all points in int$(I)$, except those on a finite set of periodic orbits or preimages of these orbits, are attracted to it,
and when $k\to 1$,  $\tilde\Lambda(k,r,p)\to  \{1-k,1\}$. More precisely, for $p\ge 1$:
\medskip

{\rm (i)} When $(k,r)\in {\rm int}(S_p)$,  $J_{p,i}\bigcap J_{p,j}$ is empty for $0\le i,j\le 2^p-1$ and 
$i\neq j$; also if $(k,r)\in {\rm int}(S_{p+1})$ and $0\le i\le 2^{p}-1$, $J_{p+1,i}$ and $J_{p+1,2^{p}+i}$ are 
subintervals of $J_{p,i}$ obtained by deleting a central open interval.  
\medskip

{\rm (ii)} When $(k,r)\in {\rm int}(S_p)$, 
the orbits of all points in int$(I)$, except those on the orbits of certain unstable periodic points 
$\tilde C_m$, $0\le m\le p-1$ or preimaqes of these orbits, land in $\tilde\Lambda(k,r,p)$.
\medskip

{\rm (iii)} if $(k,r)\in {\rm int}(S_p)\setminus S_{p+1}$, 
$f$ is chaotic on $\tilde\Lambda(k,r,p)$.
\end{theorem}

\subsection{Proof of Theorem \ref{prop5} for $g$}

First we prove this result for the map $g$ in \eqref{2}.
The lemma below is the key to the proof (and it is an important element in the next
section). In this lemma, we show that $[0,1]$ splits into three intervals $I_0\cup J\cup I_1
=[0,g(b)]\cup(g(b),b)\cup[b,1]$
such that $g(I_1)=I_0$ and $g(I_0)=I_1$. The middle open
interval $J$ contains an unstable fixed point and the orbits of all other points
in $J$ eventually land in $I_0\cup I_1$. This means that $g$ can no longer be chaotic
on $[0,1]$ but, if an additional condition is satisfied, we find that $g^2$ is chaotic
on $I_1$ and hence that $g$ is chaotic on $I_0\cup I_1$. Actually this lemma is
essentially the same as Lemma 1.1 in \cite{ITN}.

\begin{lemma}\label{lem2} Consider the map $g$ in \eqref{2} with $a=1-1/k$, $b=1-ra$
and suppose $(k,r)\in S_{1}=\{(k,r): k>1, 1/k<r\le k/(k^{2}-1)\}$. Then
\medskip

{\rm (i)} $C_0=k/(k+1)$ is an unstable fixed point of $g$;
\medskip{}

{\rm (ii)} $0<a<g(b)\le \frac{k}{k+1}\le b$, with strict inequality
when $rk^2-k-r<0$ (that is, $r<k/(k^2-1)$) and $g([b,1])=[0,g(b)]=[0,k(1-b)]$; 
\medskip

{\rm (iii)} $g^2([b,1])=[b,1]$ and
\[ g^2(x)=\begin{cases} B+R(x-b) & (b\le x\le A)\\ 
                        b+K(1-x) & (A\le x\le 1),\end{cases}\]
where $A=1-a/k$, $B=g^{4}(1)=k(1-k)+k^{2}b$, $R=k^2$, $K=rk$;
\medskip

{\rm (iv)} For each $x\in [0,1]$, $x\neq C_0=k/(k+1)$, $g^{n}(x)\in [0,g(b)]\cup[b,1]$
for sufficiently large $n>0$;
\medskip

{\rm (v)} If $r^2k^3-k-r>0$, then $g^2$ is chaotic on $[b,1]$ and 
$g$ is chaotic on $[0,g(b)]\cup[b,1]$.
\end{lemma}

\begin{proof} (i) Clearly $C_0=k/(k+1)$ is a fixed point of $g$ and it is unstable since
$g'(C_0)=-k<-1$.
\medskip{}

(ii) Clearly $a>0$. Also $b>a$ so that $g(b)=k(1-b)$ and $g([b,1])=[0,g(b)]$. Then
\[ g(b)-a=(kr-1)a>0,\quad \frac{k}{k+1}-g(b)=(k-1)\left(\frac{k}{k^2-1}-r\right)\ge 0\]
and
\[b-\frac{k}{k+1}= \frac{k-1}{k}\left(\frac{k}{k^2-1}-r\right)\ge 0\]
with strict inequality if $rk^2-k-r<0$.
\medskip

(iii) Note that on $[a,1]$,
\[ g^2(x)=\begin{cases} k(1-k)+k^2b+k^2(x-b) & (a\le x\le A=1-a/k)\\ b+rk(1-x) & (A\le x\le 1).\end{cases}\]
Now
\[ A-b=\frac{(rk-1)(k-1)}{k^2}>0.\]
So $a<b<A$. The range of $g^2$ on $[b,A]$ is $[g^2(b),1]$ and on
$[A,1]$ the range is $[b,1]$.
So $g^2([b,1])=[b,1]$ if and only if $g^2(b)\ge b$. However, using (ii),
\[ g^{2}(b)-b=k(1-k)+k^2b-b=(k^{2}-1)\left(b-\frac{k}{k+1}\right)\ge 0.\] 
Lastly note that 
$k(1-k)+k^{2}b=g^{2}(b)=g^{4}(1)$.
\medskip

(iv) Let $x\in (g(b),b)$, $x\neq C_0$. Then since $g(b)>a$, $g(x)=k(1-x)$.
So $g(x)-C_0=k(1-x)-k(1-C_0)=k(C_0-x)$. Then if $g^n(x)\in (g(b),b)$ for all $n\ge 0$,
we would have $|g^n(x)-C_0|=k^n|x-C_0|\to\infty$ as $n\to\infty$.
\medskip

(v) If $h(x)=b+(1-b)x$,
\[ (h^{-1}g^{2}h)(x)=\begin{cases} \tilde B+Rx &(0\le x\le \tilde A)\\
                                     K(1-x) &(\tilde A\le x\le 1),\end{cases}
\]
where $\tilde A=h^{-1}(A)=1-1/K$, $\tilde B=1-R\tilde A$. Then 
\[ \frac{1-\tilde B}{\tilde A}=R>1,\quad  \tilde B-\frac{1}{2-\tilde A}
=-\frac{RK^{2}-K-R}{K(K+1)}.\]
Since also $RK^{2}-K-R=k(r^2k^3-k-r)>0$. it follows from Propositions 4 and 2 in \cite{B}
that $h^{-1}g^{2}h$ is chaotic on $[0,1]$ so that $g^2$ is chaotic on $[b,1]$.
Then $g$ is chaotic on $[b,1]\cup g([b,1])=[0,g(b)]\cup[b,1]$. 

\end{proof}

In the previous lemma, we considered the behaviour of $g$ on $[0,1]$
and $g^2$ on $[b,1]$, where $b=g^2(1)$. In the lemma below, for $p\ge 1$,
we consider the behaviour of $g^{2^p}$ on $[B_p,1]$
and $g^{2^{p+1}}$ on $[B_{p+1},1]$, where $B_p=g^{2^p}(1)$.
We prove it by repeatedly applying the previous lemma.
\medskip

To state the new lemma, we need to define the sets $S_{p}$ mentioned in Theorem \ref{prop5}.
Let $r_p$, $k_p$
be defined by the recurrence relations
\[ r_{p+1}=k^2_p,\quad k_{p+1}=r_pk_p,\quad p\ge 0\]
where $r_0=r$, $k_0=k$. For $p\ge 1$, write
\[ S_{p}=\{(k,r): k>1, r>1/k, k_{p-1}^2r_{p-1}-k_{p-1}-r_{p-1}\le 0\}.\]
When $k>1$ and $kr>1$, it is easy to see that $k_p>1$, $r_p>1$ for $p\ge 1$ and hence that
$k_pr_p>1$ for $p\ge 0$. Then 
\[r_{p+1}k^{2}_{p+1}-k_{p+1}-r_{p+1}=k_{p}(r^2_{p}k^3_{p}-k_{p}-r_{p})
>k_{p}(r_{p}k^2_{p}-k_{p}-r_{p}).\]
So $r_{p+1}k^{2}_{p+1}-k_{p+1}-r_{p+1}\le 0$ implies that $r_{p}k^{2}_{p}-k_{p}-r_{p}<0$. 
Since it is easy to
see by induction that $r_p$ and $k_p$ are continuous functions of $(k,r)$ so that
\[ {\rm int}(S_{p})=\{(k,r): k>1, r>1/k, k_{p-1}^2r_{p-1}-k_{p-1}-r_{p-1}< 0\},\]
it follows that $S_{p+1}\subset$int$(S_p)$ for $p\ge 1$.

\begin{lemma} \label{lem3} Consider the map $g$ in \eqref{2} with $a=1-1/k$, $b=1-ra$.
 For $p\ge 0$, define
\begin{equation}\label{ABC} B_p=g^{2^p}(1),\quad A_p=1-(1-B_p)/k_p,\quad 
     C_p=\frac{B_{p}+k_{p}}{k_{p}+1}.\end{equation}
Then for $p\ge 0$, if $(k,r)\in S_{p+1}$, 
\medskip

{\rm (i)} $C_{p}$ is an unstable fixed point of $g^{2^{p}}$;
\medskip

{\rm (ii)} 
\[B_{p}<A_{p}<g^{2^{p}}(B_{p+1})\le C_{p}\le B_{p+1}\]
with the last two inequalities strict when $(k,r)\in{\rm int}(S_{p+1})$ and 
\[g^{2^{p}}([B_{p+1},1])=[B_{p},g^{2^{p}}(B_{p+1})];\]

{\rm (iii)} $g^{2^{p+1}}([B_{p+1},1])=[B_{p+1},1]$
and
\[ g^{2^{p+1}}(x)=\begin{cases} B_{p+2}+r_{p+1}(x-B_{p+1}) &(B_{p+1}\le x\le A_{p+1})\\
                            B_{p+1}+k_{p+1}(1-x) &(A_{p+1}\le x\le 1).\end{cases}
\]

{\rm (iv)} For each $x\in [B_{p},1]$, $x\neq C_{p}$, 
$g^{n2^{p}}(x)\in [B_{p},g^{2^{p}}(B_{p+1})]\cup[B_{p+1},1]$
for sufficiently large $n>0$.
\medskip

{\rm (v)} When $(k,r)\in{\rm int}(S_{p+1})$, 
the intervals $g^i([B_{p+1},1])$, $i=0,\ldots, 2^{p+1}-1$,
are disjoint and for $i=0,\ldots,2^p-1$, $g^i([B_{p+1},1])$ and $g^{2^p+i}([B_{p+1},1])$ are
contained in $g^i([B_{p},1])$ and are obtained by removing a central open
interval, which contains $g^i(C_p)$, from $g^i([B_{p},1])$.
\medskip

{\rm (vi)} When $(k,r)\in{\rm int}(S_{p+1})$, the orbits of all points in $[0,1]$, 
except for the orbits of $C_q$,
$0\le q\le p$, and their preimages, eventually land in the union $\Lambda(k,r,p+1)$ of
the disjoint intervals $g^i([B_{p+1},1])$, $i=0,\ldots, 2^{p+1}-1$. 
\medskip

{\rm (vii)} If $(k,r)\in{\rm int}(S_{p+1})\setminus S_{p+2}$, then 
$g^{2^{p+1}}$ is chaotic on $[B_{p+1},1]$ and therefore $g$ is chaotic on 
$\Lambda(k,r,p+1)$. 
\end{lemma}

\begin{proof} For $p\ge 0$, if $r_{p}k^{2}_{p}-k_{p}-r_{p}\le 0$, we prove (i)-(vii) hold
by induction on $p$. 
\medskip

For $p=0$, the statements (i)-(vii) follow from Lemma \ref{lem2}.
\medskip

Assuming the statements (i)-(vii) are true for $p\ge 0$, we prove them for $p+1$.
So we are assuming $r_{p+1}k^{2}_{p+1}-k_{p+1}-r_{p+1}\le 0$, which implies
that $r_{p}k^{2}_{p}-k_{p}-r_{p}< 0$.
\medskip

So by the induction hypothesis,
\[ g^{2^{p+1}}([B_{p+1},1])=[B_{p+1},1]\]
and
\[ g^{2^{p+1}}(x)=\begin{cases} B_{p+2}+r_{p+1}(x-B_{p+1}) &(B_{p+1}\le x\le A_{p+1})\\
                            B_{p+1}+k_{p+1}(1-x) &(A_{p+1}\le x\le 1).\end{cases}
\]
Observe since 
\[g^{2^{p+1}}(A_{p+1})=B_{p+2}+r_{p+1}(A_{p+1}-B_{p+1})
=B_{p+1}+k_{p+1}(1-A_{p+1})=1,\] 
it follows on elimination of $A_{p+1}$ that
\begin{equation} \label{speqn}(1-B_{p+2})=r_{p+1}(1-1/k_{p+1})(1-B_{p+1}).\end{equation}
Now if $h(x)=B_{p+1}+(1-B_{p+1})x$,
\[ G(x)=(h^{-1}g^{2^{p+1}}h)(x)=\begin{cases} b_{p+1}+r_{p+1}x &(0\le x\le a_{p+1})\\
                                     k_{p+1}(1-x) &(a_{p+1}\le x\le 1),\end{cases}
\]
where  
\begin{equation}\label{bp} a_{p+1}=h^{-1}(A_{p+1}),\quad b_{p+1}=h^{-1}(B_{p+2})\end{equation}
but also 
\[a_{p+1}=1-1/k_{p+1},\quad b_{p+1}=1-r_{p+1}a_{p+1}\]
since 
\[ G(a_{p+1})=(h^{-1}g^{2^{p+1}}h)(a_{p+1})=h^{-1}(g^{2^{p+1}}(A_{p+1}))=h^{-1}(1)=1.\]
$G$ satisfies the conditions of Lemma \ref{lem2} since $k_{p+1}>1$,
$k_{p+1}r_{p+1}>1$ and $r_{p+1}k^{2}_{p+1}-k_{p+1}-r_{p+1}\le 0$.
\medskip

Now we prove (i) for $p+1$. We know from Lemma \ref{lem2} (i) applied to $G$ that $k_{p+1}/(k_{p+1}+1)$ is
an unstable fixed point of $G$. It follows that
\begin{equation}\label{cpeq} h(k_{p+1}/(k_{p+1}+1))
=(B_{p+1}+k_{p+1})/(k_{p+1}+1)=C_{p+1}\end{equation}
is an unstable fixed point of $hGh^{-1}=g^{2^{p+1}}$. So (i) is proved for $p+1$.
\medskip

Now we prove (ii) for $p+1$. We know from Lemma \ref{lem2} (ii) applied to $G=h^{-1}g^{2^{p+1}}h$
that $G([b_{p+1},1])=[0,G(b_{p+1})]$ and so, using \eqref{bp},
\[\begin{array}{rl}
g^{2^{p+1}}([B_{p+2},1])&=(hGh^{-1})([B_{p+2},1])=(hG)([b_{p+1},1])\\ \\
&=h([0,G(b_{p+1})])=[B_{p+1},g^{2^{p+1}}(h(b_{p+1}))]\\ \\
&=[B_{p+1},g^{2^{p+1}}(B_{p+2})].\end{array}\] 
Also, again from Lemma \ref{lem2} (ii),
\[ 0<a_{p+1}<G(b_{p+1})\le \frac{k_{p+1}}{k_{p+1}+1}\le b_{p+1},\]
with strict inequality when $r_{p+1}k^{2}_{p+1}-k_{p+1}-r_{p+1}<0$, so that, applying $h$,
\[ B_{p+1}<A_{p+1}<g^{2^{p+1}}(B_{p+2})\le  \frac{B_{p+1}+k_{p+1}}{k_{p+1}+1} \le B_{p+2}\]
with strict inequality when $r_{p+1}k^{2}_{p+1}-k_{p+1}-r_{p+1}<0$. This proves (ii) for $p+1$.
\medskip

Next we prove (iii) for $p+1$. By Lemma \ref{lem2} (iii) applied to $G$,  
\begin{equation}\label{Geq}G^2([b_{p+1},1])=[b_{p+1},1]\end{equation}
and
\[ G^2(x)=\begin{cases} B+R(x-b_{p+1}) & (b_{p+1}\le x\le A)\\ 
                        b_{p+1}+K(1-x) & (A\le x\le 1),\end{cases}\]
where $A=1-a_{p+1}/k_{p+1}$,
\begin{equation}\label{newone} B=G^2(b_{p+1})=(h^{-1}g^{2^{p+2}}h)(b_{p+1})=h^{-1}(g^{2^{p+2}}(B_{p+2}))=h^{-1}(B_{p+3}),\end{equation} 
$R=k_{p+1}^2=r_{p+2}$, $K=r_{p+1}k_{p+1}=k_{p+2}$. Then \eqref{Geq} implies
\[ g^{2^{p+2}}([B_{p+2},1])=(hG^2)([b_{p+1},1])=h([b_{p+1},1])=[B_{p+2},1],\]
and using \eqref{bp} and \eqref{newone},
\[ g^{2^{p+2}}(x)=(hG^2h^{-1})(x)
=\begin{cases} B_{p+3}+r_{p+2}(x-B_{p+2}) & (B_{p+2}\le x\le h(A))\\ 
                        B_{p+2}+k_{p+2}(1-x) & (h(A)\le x\le 1),\end{cases}\]
where, using  \eqref{speqn}, $h(A)=B_{p+1}+(1-B_{p+1})A=1-(1-B_{p+1})a_{p+1}/k_{p+1}
=1-(1-B_{p+2})/(r_{p+1}k_{p+1})=1-(1-B_{p+2})/k_{p+2}=A_{p+2}$.
\medskip

This completes the proof of (iii) for $p+1$.
\medskip

We also know from Lemma \ref{lem2} (iv) applied to $G=h^{-1}g^{2^{p+1}}h$ that if $x\in [0,1]$,
$x\neq k_{p+1}/(k_{p+1}+1)$, then $G^n(x)\in [0,G(b_{p+1})]\cup [b_{p+1},1]$
for some $n\ge 0$. This means that if $x\in [B_{p+1},1]$ and 
$x\neq h(k_{p+1}/(k_{p+1}+1))=C_{p+1}$ (see \eqref{cpeq}), 
then $g^{n2^{p+1}}(x)\in [B_{p+1},g^{2^{p+1}}(B_{p+2})]\cup [B_{p+2},1]$
for some $n\ge 0$. This proves (iv) for $p+1$.
\medskip

Now we show (v) for $p+1$. By the induction hypothesis, the intervals $g^i([B_{p+1},1])$, 
$i=0,\ldots, 2^{p+1}-1$ are disjoint. Then for $i=0,\ldots, 2^{p+1}-1$,
\[\begin{array}{l}
g^i([B_{p+2},1])\subset g^i([B_{p+1},1]),\\ \\ 
g^{2^{p+1}+i}([B_{p+2},1])=g^i([B_{p+1},g^{2^{p+1}}(B_{p+2})])\subset g^i([B_{p+1},1]).
\end{array}\]
Since $g^{2^{p+1}}([B_{p+1},1])=[B_{p+1},1]$, we have $a\in g^{2^{p+1}-1}([B_{p+1},1])$. 
Thus
$g$ is one to one on $g^i([B_{p+1},1])$ for $0\le i<2^{p+1}-1$
and hence $g^i$ is one to one on 
\[ [B_{p+1},1]=[B_{p+1},g^{2^{p+1}}(B_{p+2})]\cup(g^{2^{p+1}}(B_{p+2}),B_{p+2})\cup[B_{p+2},1] \]
for the same $i$. It follows that $g^i([B_{p+2},1])$ and 
\[g^{2^{p+1}+i}([B_{p+2},1])=g^i([B_{p+1},g^{2^{p+1}}(B_{p+2})])\]
are contained in $g^i([B_{p+1},1])$, are disjoint and each contains an
endpoint of $g^i([B_{p+1},1])$. Hence $g^i([B_{p+2},1])$ and $g^{2^{p+1}+i}([B_{p+2},1])$
are what remains after removing a central open interval from $g^i([B_{p+1},1])$.
Also since $g^{2^{p+1}}(B_{p+2})< C_{p+1}<B_{p+2}$
when $r_{p+1}k^{2}_{p+1}-k_{p+1}-r_{p+1}<0$ and $g^i$ is one to one on $[B_{p+1},1]$ for
$0\le i <2^{p+1}-1$, it follows that
$g^i(C_{p+1})$ is in the central open interval for $0\le i \le 2^{p+1}-1$. Then (v) follows for $p+1$.
\medskip

Now we prove (vi) for $p+1$. Suppose $x\in [0,1]$ does not lie on the orbit of $C_m$ for $0\le m\le p+1$.
Then by the induction hypothesis, there is some $n\ge 0$ such that  
$g^n(x)\in \Lambda(k,r,p+1)$. Hence there exists $i$, $0\le i<2^{p+1}$,
such that $g^n(x)\in g^i([B_{p+1},1])$ so that
$g^{n+2^{p+1}-i}(x)\in g^{2^{p+1}}([B_{p+1},1])=[B_{p+1},1]$.
Since $g^{n+2^{p+1}-i}(x)\neq C_{p+1}$, it follows from (iv) for $p+1$ as we have just proved
that there exists $m>0$ such that 
\[ g^{m2^{p+1}+n+2^{p+1}-i}(x)\in [B_{p+1},g^{2^{p+1}}(B_{p+2})]\cup[B_{p+2},1]
=g^{2^{p+1}}([B_{p+2},1])\cup[B_{p+2},1]\]
which is contained in $\Lambda(k,r,p+2)$.
This proves (vi) for $p+1$.
\medskip

Finally we show (vii) for $p+1$. 
We know from Lemma \ref{lem2} (v) applied to $G^2$ that if $RK^2-K-R>0$,
then $G^2$ is chaotic on $[b_{p+1},1]$. So if $r_{p+1}k^{2}_{p+1}-k_{p+1}-r_{p+1}>0$, 
$g^{2^{p+2}}$ is chaotic on $h([b_{p+1},1])=[B_{p+2},1]$. It follows
that $g$ is chaotic on $\Lambda(k,r,p+2)$.
\medskip

This completes the induction proof that (i)-(vii) hold for $p\ge 1$.
\end{proof}

\begin{remark} \label{rem1} It follows from (v) and (vi) in Lemma \ref{lem3}
	that if $p\ge 0$ and $(k,r)\in{\rm int}(S_{p+1})$ then $\Lambda(k,r,p+1)\subset \Lambda(k,r,p)$.
	So for $p\ge 1$ and $(k,r)\in{\rm int}(S_{p})$, $\Lambda(k,r,p)\subset 
\Lambda(k,r,1)=I_{1,0}\cup I_{1,1}=[0,k(1-b)]\cup[b,1]=[0,rk-r]\cup[1-r+r/k,1]$
so that $\Lambda(k,r,p)\to \{0,1\}$ as $k\to 1$.
\end{remark}

\subsection{Proof of Theorem \ref{prop5}}

\begin{proof} First it follows from Lemma \ref{lem1} that if $x\in$ int$(I)$,
that there exists $m\ge 0$ such that $f^{n}(x)\in [1-k,1]$ for $n\ge m$.
Also $f$ maps $[1-k,1]$ into itself and on $[1-k,1]$,
$f=hgh^{-1}$, where $g:[0,1]\to [0,1]$ is as in \eqref{2} and $h(x)=1-k+kx$.
\medskip{}

Now suppose $p\ge 1$ and $(k,r)\in {\rm int}(S_{p})$. Then
\[ h(g^{i}([B_{p},1]))=f^{i}(h([B_{p},1])),\]
where, using \eqref{ABC},
\begin{equation}\label{beq} B_{p}=g^{2^{p}}(1)=h^{-1}(f^{2^{p}}(h(1)))=h^{-1}(f^{2^{p}}(1))\end{equation}
so that 
\[ h([B_{p},1])=[h(B_{p}),1]=[f^{2^{p}}(1),1].\]
Thus
\[ h(g^{i}([B_{p},1]))=f^i(h([B_{p},1]))=f^{i}([f^{2^{p}}(1),1])=J_{p,i}.\]
Then (i) follows from (v) in Lemma \ref{lem3}. 
\medskip{}

Now we see that 
\[\tilde\Lambda(k,r,p)=\bigcup_{i=0}^{2^{p}-1}J_{p,i}=\bigcup_{i=0}^{2^{p}-1}h(g^i([B_p,1]))
= h(\Lambda(k,r,p)), \]
where $\Lambda(k,r,p)$ is as in Lemma \ref{lem3}.
Since $\Lambda(k,r,p)$ is invariant under $g$, $\tilde\Lambda(k,r,p)$ is invariant under $f$,
and since from Remark \ref{rem1}, $\Lambda(k,r,p)\to \{0,1\}$ as $k\to 1$, it follows that
$\tilde\Lambda(k,r,p)\to h(\{0,1\})=\{1-k,1\}$ as $k\to 1$.
\medskip{}

Since the $g-$orbits of all points in $[0,1]$ except those on the orbits of the unstable periodic points $C_m$,  $0\le m\le p-1$, or their preimages, land in $\Lambda(k,r,p)$, it follows that
the $f-$orbits of all points in $[1-k,1]$ except those on the orbits of the unstable periodic points $\tilde C_m=h(C_{m})$,  $0\le m\le p-1$, or their preimages, land in $\tilde\Lambda(k,r,p)$.
Note, using \eqref{ABC} and \eqref{beq},
\[ h(C_{m})=1-k+kC_{m}=1-k+k\frac{B_{m}+k_{m}}{k_{m}+1}=\frac{f^{2^{m}}(1)+k_{m}}{k_{m}+1}.\]
Thus (ii) is proved.
\medskip{}

Finally (iii) follows directly from (vii) in Lemma \ref{lem3}. 
\end{proof}

\begin{remark}
In \cite{SAG}, Proposition 4.6 on page 612, the 
band attractors are described but it is not shown as we do here
that all points in int$(I)$, except those which  
lie on a finite set of periodic orbits or are preimages of these orbits, 
are attracted to the attractor. 
\medskip

In Section 5 in \cite{B}, where this parameter range is considered,
it is only proved that some power of the map exhibits chaos on some
interval and that there are no attracting periodic orbits.
The region is not divided into infinitely many subregions each
with a chaotic band attractor. The same comment applies to \cite{LT}.
In fact, in \cite{LT} this region is not separately considered.
\medskip

The region int$(S_p)\setminus S_{p+1}$ corresponds to $D^{(p+1)}_0$
in \cite{ITN}.
\end{remark}

\subsection{Geometry of the regions $S_{p}$}

Now we describe the geometry of the regions $S_p$. First we
derive formulae for $r_p$ and $k_p$.

\begin{lemma}\label{lem5}  The recurrence relations
\[ r_{p+1}=k^2_p,\quad k_{p+1}=r_pk_p,\]
where $r_0=r$, $k_0=k$, have the solution
\[ k_p=r^{\chi(p)/2}k^{\chi(p)+(-1)^p},\quad 
   r_p=r^{\chi(p)/2+(-1)^p}k^{\chi(p)},\]
with $\chi(p)=(2^{p+1}+2(-1)^{p+1})/3$. Also for $p\ge 0$,
\[ r_pk^2_p-k_p-r_p
=\begin{cases}
r^{\chi_(p)/2-1}k^{\chi(p)-1}[r^{\chi(p)}k^{2\chi(p)-1}-k-r]& (p\; {\rm odd})\\ \\
r^{\chi_(p)/2}k^{\chi(p)}[r^{\chi(p)+1}k^{2\chi(p)+2}-k-r]& (p\; {\rm even}).\end{cases}\]
\end{lemma}

\begin{proof} 
If we define $x_p=r_p/k_p$, we see that $x_{p+1}=x^{-1}_p$
and so $x_p=x_0^{(-1)^p}=(r/k)^{(-1)^p}$. Then
\[ k_{p+1}=x_pk^2_p=\left(\frac{r}{k}\right)^{(-1)^p}k^2_p,\]
which we solve as
\[ k_p=r^{\chi(p)/2}k^{\chi(p)+(-1)^p}.\]
Then
\[ r_p=x_pk_p=r^{\chi(p)/2+(-1)^p}k^{\chi(p)}.\]
Next 
\[ \begin{array}{rl}
&r_pk^2_p-k_p-r_p\\ \\
&=k_{p}(r_{p}k_{p}-x_{p}-1)\\ \\
&=r^{\chi_(p)/2}k^{\chi(p)+(-1)^p}[r^{\chi(p)+(-1)^p}k^{2\chi(p)+(-1)^p}
-r^{(-1)^p}k^{(-1)^{p+1}}-1]\\ \\
&=\begin{cases} r^{\chi(p)/2}k^{\chi(p)-1}[r^{\chi(p)-1}k^{2\chi(p)-1}
-r^{-1}k-1]& (p\; {\rm odd})\\
               r^{\chi(p)/2}k^{\chi(p)+1}[r^{\chi(p)+1}k^{2\chi(p)+1}
-rk^{-1}-1]& (p\; {\rm even})\end{cases}\\ \\
&=\begin{cases} r^{\chi(p)/2-1}k^{\chi(p)-1}[r^{\chi(p)}k^{2\chi(p)-1}
-k-r]& (p\; {\rm odd})\\
               r^{\chi(p)/2}k^{\chi(p)}[r^{\chi(p)+1}k^{2\chi(p)+2}
-k-r]& (p\; {\rm even}).\end{cases}
\end{array}\]

\end{proof}

It follows from this lemma that $r_pk^2_p-k_p-r_p$ has the same sign as the polynomial
\begin{equation}\label{tpeq}
	 t_{p}(k,r)=\begin{cases} r^{\chi(p)}k^{2\chi(p)-1}-k-r& (p\; {\rm odd})\\
               r^{\chi(p)+1}k^{2\chi(p)+2}-k-r& (p\; {\rm even}).\end{cases}\end{equation}
Next for $p\ge 0$, we determine the sign of the polynomials $t_p(k,r)$,
and hence the sign of $r_pk^2_p-k_p-r_p$
in the region $S_1$. Thus we determine the $S_{p}$.
\medskip{}

\begin{lemma}\label{lem6} Let $t_p(k,r)$ be as in \eqref{tpeq} with $p\ge 0$.
\medskip{}

{\rm (i)} for $p\ge 0$, there is a strictly decreasing function $\rho_{p}(k)>0$,
defined for $k>1$ if $p=0$ and for $k\ge 1$ if $p\ge 1$,
such that $t_{p}(k,r)$ has the same sign as $r-\rho_{p}(k)$ 
in $k>1$, $r>0$,
where
\[ \rho_0(k)=\frac{k}{k^2-1}>\frac{1}{k},\quad
 \rho_1(k)=\frac{1+\sqrt{1+4k^4}}{2k^3}>\frac{1}{k},\quad k>1;\]
{\rm (ii)} there is a decreasing sequence $K_{p}$, $p\ge 2$, with $1<K_p<2$ and
tending to $1$ as $p\to\infty$ such that 
\[ \rho_{p}(k)\begin{cases} >1/k & (1\le k<K_{p}),\\
                            =1/k & (k=K_{p}),\\
                            <1/k & (k>K_{p}).\end{cases}\]
                            
{\rm (iii)} $\rho_{p+1}(k)<\rho_{p}(k)$ for $1<k\le K_{p+1}$, $p\ge 0$
with $\le K_1$ interpreted as $<\infty$.
\medskip{}

{\rm (iv)} for $p\ge 1$, 
\[S_{p}=\{(k,r): 1<k<K_{p-1}, 1/k<r\le \rho_{p-1}(k)\},\]
where $K_0=K_1=\infty$.
\end{lemma}

\begin{proof}
	(i) For $p=0$, since $t_0(k,r)=rk^2-k-r$, the statement is true.
For $p=1$,
\[ t_{1}(k,r)=k^{3}r^2-r-k=k^3(r-(1+\sqrt{1+4k^4})/(2k^3))(r+(\sqrt{1+4k^4}-1)/(2k^3))\]
and hence $t_1(k,r)$, has the same sign as $r-\rho_{1}(k)$.
It is easy to see that $\rho_0(k)>1/k$ and $\rho_1(k)>1/k$.
\medskip{}

	For $p\ge 2$, note that for fixed $k>1$,  $t_p(k,r)$ is a polynomial in $r$
which is strictly convex in $r$ in $r\ge 0$ and tends to $\infty$
as $r\to\infty$. Moreover when $r=0$, it is negative.
It follows that for $p\ge 2$ and fixed $k>1$,  there is a well-defined function
$\rho_{p}(k)>0$ such that $t_{p}(k,r)=0$ in $r\ge 0$ if and only if $r=\rho_{p}(k)$. 
\medskip

Next when $p\ge 2$ is even and $r=\rho_{p}(k)$,
\[\frac{\partial}{\partial r}t_p(k,r)=\chi(p)+(\chi(p)+1)k/r>0,\]
and
 \[ \frac{\partial}{\partial k}t_p(k,r)=2\chi(p)+1+(2\chi(p)+2)r/k>0.\]
When $p\ge 3$ is odd and $r=\rho_{p}(k)$,
\[\frac{\partial}{\partial r}t_p(k,r)=\chi(p)(1+k/r)-1>0,\]
and
 \[ \frac{\partial}{\partial k}t_p(k,r)=2\chi(p)-2+(2\chi(p)-1)r/k>0.\]
From the implicit function theorem, it follows that $\rho'_{p}(k)$ exists and is negative 
so that $\rho_{p}(k)$ is strictly decreasing. Also since both derivatives with
respect to $r$ are positive at $r=\rho_{p}(k)$, $t_{p}(k,r)$ has the same sign
as $r-\rho_{p}(k)$.
 \medskip{}

(ii) Now we have $p\ge 2$. Since $t_{p}(1,1)=-1$, $t_{p}(2,1/2)>0$ and $\frac{d}{dk}t_p(k,1/k)>0$
for $k\ge 1$, $t_p(k,1/k)$ has a unique zero $K_p$ in $k\ge 1$ which lies in $(1,2)$. Then
$t_p(k,1/k)$ has the same sign as $k-K_{p}$.
Since $t_{p+1}(k,1/k)>t_{p}(k,1/k)$ in $k\ge 1$, it follows that $K_{p+1}<K_{p}$.
\medskip

Now if $k<K_{p}$ (resp. $=, >$), then $t_p(k,1/k)<0$ (resp. $=, >$) 
which implies that $1/k<\rho_{p}(k)$ (resp. $=, >$).  
\medskip

(iii) With $r=\rho_{p+1}(k)$ where $1<k\le K_{p+1}$
with $\le K_1$ interpreted as $<\infty$, $t_{p+1}(k,r)=0$ implies that $r_{p+1}k^2_{p+1}-k_{p+1}-r_{p+1}=0$, which 
by the remarks before Lemma \ref{lem3}, implies that
$r_pk^2_p-k_p-r_p<0$. By Lemma \ref{lem5}, $r_pk^2_p-k_p-r_p$ has the same sign as $t_p(k,r)$ and hence
the same sign as $r-\rho_p(k)$ which must therefore be negative.
That is, $\rho_{p+1}(k)-\rho_p(k)<0$.
\medskip

(iv) For $p=1$, we have the formula in \eqref{s0eq} where $\rho_0(k)=k/(k^2-1)$. For $p\ge 2$,
\[ \begin{array}{rl}
S_{p}
&=\{(k,r): k>1,\; r>1/k,\; k_{p-1}^2r_{p-1}-k_{p-1}-r_{p-1}\le 0\}\\ \\
&=\{(k,r): k>1,\; r>1/k,\; t_{p-1}(k,r)\le 0\}\\ \\
&=\{(k,r): k>1,\; r>1/k,\; 0< r\le \rho_{p-1}(k)\}\\ \\
&=\{(k,r): k>1,\; 1/k< r\le \rho_{p-1}(k)\}\\ \\
&=\{(k,r): 1<k<K_{p-1},\; 1/k< r\le \rho_{p-1}(k)\},
\end{array}\]
where $K_1=\infty$. 
\end{proof}

\begin{remark}
In fact we can solve the cubic $k^6r^3-k-r=0$ to get
\[ \rho_2(k)=\left(\frac{1+\sqrt{1-\frac{4}{27k^{8}}}}{2k^5}\right)^{1/3}
                +\left(\frac{1-\sqrt{1-\frac{4}{27k^{8}}}}{2k^5}\right)^{1/3}.\]
\end{remark}

\begin{figure}
\centering
\includegraphics[width=0.80\textwidth]{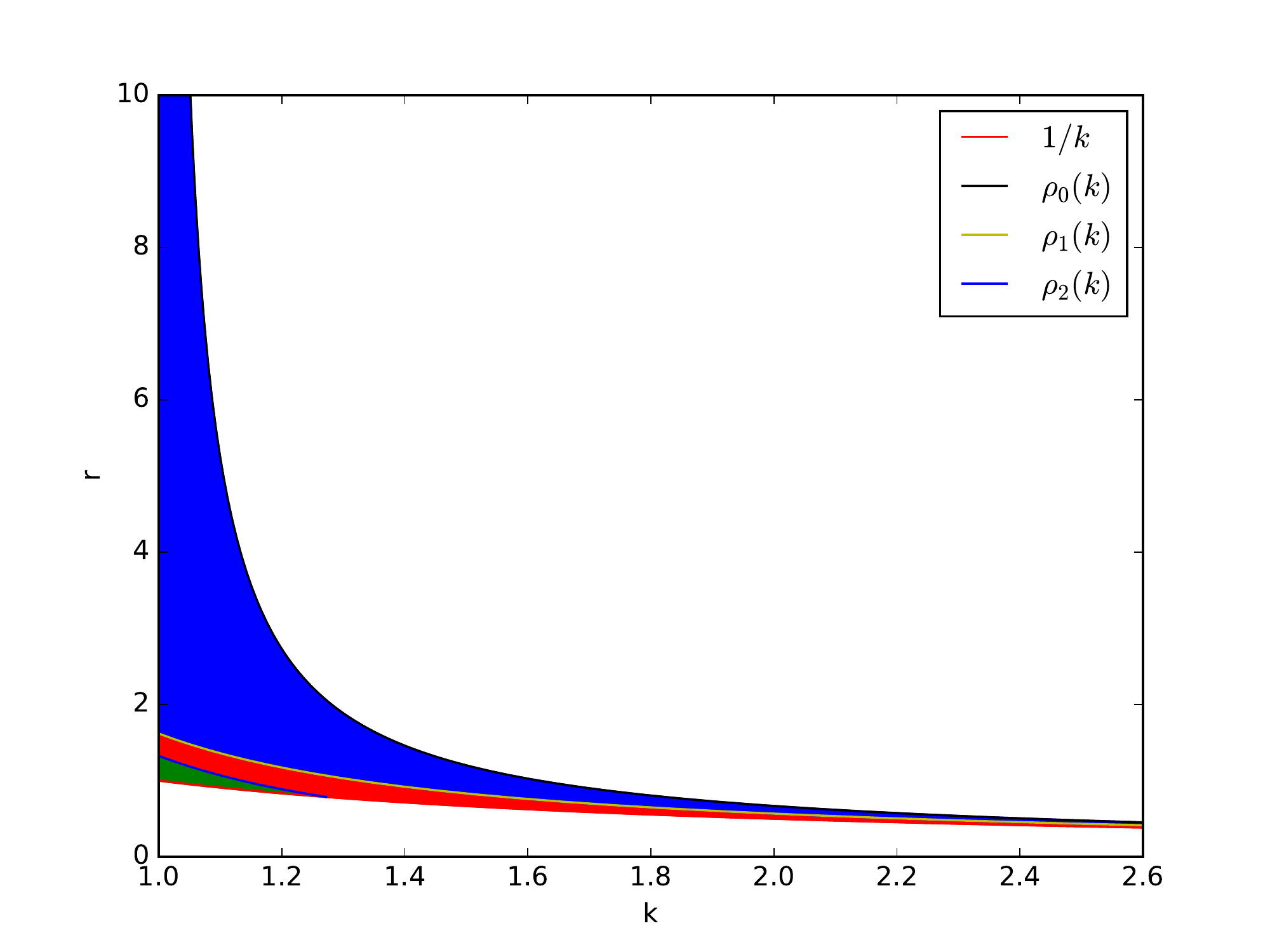}
\caption{Regions $S_1\setminus S_2$ (blue), $S_2\setminus S_3$ (red), $S_3$ (green) in $(k,r)-$parameter space.
	In $S_1\setminus S_2$ the chaotic attractor consists of two intervals, in $S_2\setminus S_3$ of four and in $S_3$ of eight or more.}   
\end{figure}

Redwood City
\section{Attracting periodic orbits, chaotic band attractors and chaotic Cantor sets}

Suppose $f$ is as in \eqref{1} with 
\begin{equation}\label{reg} (k,r)\in \{(k,r): 0<r<1,\; k>1+1/r\}.\end{equation}

\begin{remark}
The region \eqref{reg} in the parameter space corresponds to the region defined in equation (13) in \cite{SAG}.
However equation (13) appears to be wrong as it implies that both slopes
have absolute values $>1$, inconsistent with the existence of stable periodic 
orbits. In fact, the region in (13) should be
\[ \{(a_{{\cal L}},a_{{\cal R}}): 0<a_{{\cal L}}<1, 
                  a_{{\cal R}}<-(a_{{\cal L}}+1)/a_{{\cal L}}\}\]
which in our notation is \eqref{reg}.                  
Note that the $n$ used in \cite{SAG} equals our $m+1$, where $m$ is defined
below. This region is studied in Section 6 in \cite{B}.
\end{remark}

Corresponding to each integer $m\ge 2$, we define 
\begin{equation}\label{altkm} K_m(r)=1+1/r+1/r^{2}+\cdots+1/r^{m-1}
=\frac{1-r^m}{r^{m-1}(1-r)}.\end{equation}
Then the sets
\[ \{(k,r): 0<r<1, K_m(r)<k\le K_{m+1}(r)\},\quad m\ge 2\]
partition the space $\{(k,r):0<r<1, k>1+1/r\}$. We define
\[ T_m=\{(k,r): 0<r<1, K_m(r)<k<K_{m+1}(r)\},\quad m\ge 2.\]
(Note that $T_m$ is called $D_m$ in \cite{ITN}.)
We study separately the behaviour of the map in these 
subranges. Now we state the main theorem proved in this section.

\begin{theorem}\label{thm2} Let $f$ be as in \eqref{1} and let $m\ge 2$. Then $T_{m}$ can be divided into
four subranges 
\[\begin{array}{l}   R_{m1}=\{(k,r)\in T_m: kr^m<1\},\\ \\
 R_{m2}=\{(k,r)\in T_m: kr^m>1,\; r^mk^2-k-r<0,\; r^{2m}k^3-k-r>0\},\\ \\
R_{m3}=\{(k,r)\in T_m: kr^m>1,\; r^mk^2-k-r<0,\; r^{2m}k^3-k-r<0\},\\ \\ R_{m4}=\{(k,r)\in T_m: kr^m>1,\;r^mk^2-k-r>0\},\end{array}\]
in each of which $f$ has an attractor. In $R_{m1}$ and $R_{m2}$,
the orbits of all points not lying in a certain chaotic
Cantor set or preimages of this Cantor set eventually go to the attractor;  in $R_{m3}$
the orbits of all points not lying in a certain chaotic
Cantor set or on a certain unstable periodic orbit or preimages of this Cantor set
or periodic orbit eventually go to the attractor; in $R_{m4}$ the orbits of all points eventually lie in the attractor. In $R_{m1}$ the attractor is a
periodic orbit with period $m+1$, in $R_{m4}$ it is the interval $[1-k,1]$
on which the dynamics is chaotic and in $R_{m2}$ and $R_{m3}$,
the attractor is also chaotic and consists respectively of $m+1$
and $2m+2$ disjoint closed intervals.
\end{theorem}

From Lemma \ref{lem1},
we know that for all real $x$, there exists $m\ge 0$ such that
$f^{n}(x)\in [1-k,1]$ for $n\ge m$. Then if we restrict $f$ to $[1-k,1]$,
the map $g=h^{-1}fh:[0,1]\to[0,1]$, where $h(x)=1-k+kx$,
has the form \eqref{2} with
\[ a=1-1/k,\quad b=1-ra.\]
Note that when $0<r<1$, $k>1+1/r$,
\[ b<a.\]
In subsections 5.1 to 5.5, we prove the results in Theorem \ref{thm2} for the map $g$. 
Theorem \ref{thm2} is proved in Section 5.6. The geometry of 
the $R_{mi}$ is described in Section 5.7.

\subsection{Properties of the iterates of $g$ in \eqref{2} when $(k,r)\in T_m$}
In this subsection we derive a formula for $g^{m+1}$. First we introduceRedwood City
the important quantity $x_m$.
\medskip

If $0<r<1$ and $K_m(r)<k <K_{m+1}(r)$, then
\[ r^m<\frac{1}{r+k(1-r)}< r^{m-1}\]
and hence
\[ \frac{1-r^{m-1}}{1-r}b< a<\frac{1-r^{m}}{1-r}b\]
and therefore
\[ (1+r+\cdots+r^{m-2})b<a<(1+r+\cdots+r^{m-1})b.\]
Then we see that for $1\le i\le m-1$,
\begin{equation}\label{yeq1} g^{i}(0)=(1+r+\cdots+r^{i-1})b < a\end{equation}
but 
\begin{equation}\label{yeq2} g^{m}(0)=(1+r+\cdots+r^{m-1})b > a.\end{equation}
\medskip

Note by \eqref{yeq1} and \eqref{yeq2},
there exists $x_m$ with $0<x_m<b<a$ such that
\begin{equation}\label{rel3} g^{m-1}(x_{m})=(1+r+\cdots +r^{m-2})b+r^{m-1}x_m=a.\end{equation}
Then since $b=1-ra$,
\begin{equation}\label{rel4} r^{m}x_{m}=1-(1+r+\cdots +r^{m-1})b\end{equation}
and since $a=1-1/k$, it further follows that
\begin{equation}\label{rel5} x_{m}=1-\frac{1-r^{m}}{kr^{m-1}(1-r)}.\end{equation}
$x_m$ plays an important role in the sequel.
\medskip

Next we derive a formula for the first three segments of $g^{m+1}(x)$ when $(k,r)\in T_m$.
\begin{lemma} \label{lem7} Let $(k,r)\in T_m$. Then
\[ g^{m+1}(x)=\begin{cases} -kr^m(x-x_m) & (0\le x\le x_m)\\ \\
               k^2r^{m-1}(x-x_m) & (x_m\le x\le x_m+{1\over k^2r^{m-1}})\\ \\
               b+r-kr^m(x-x_m) & (x_m+{1\over k^2r^{m-1}}\le x \le b+rx_m).\end{cases}\]
\end{lemma}

\begin{proof}          
             
Note that if $0\le x\le x_m$, using \eqref{rel3},
\[  (1+r+\cdots +r^{m-2})b+r^{m-1}x\le (1+r+\cdots +r^{m-2})b+r^{m-1}x_m=a.\]
So $(1+r+\cdots +r^{i-1})b+r^{i}x\le a$ if $1\le i\le m-1$
and hence 
\begin{equation}\label{rel1} g^i(x)=(1+r+\cdots +r^{i-1})b+r^{i}x\quad{\rm if}\quad 1\le i\le m,\; 0\le x\le x_m.
\end{equation}
Similarly if $0\le x\le b+rx_m$,
\[  (1+r+\cdots +r^{m-3})b+r^{m-2}x\le (1+r+\cdots +r^{m-2})b+r^{m-1}x_m=a.\]
So $(1+r+\cdots +r^{i-1})b+r^{i}x\le a$ if $1\le i\le m-2$
and hence 
\begin{equation}\label{rel2} g^i(x)=(1+r+\cdots +r^{i-1})b+r^{i}x\quad{\rm if}\quad 1\le i\le m-1,\; 
0\le x\le b+rx_m.\end{equation}
Next if $x_m\le x\le b+rx_m$,
\[ g^{m-1}(x)=(1+r+\cdots +r^{m-2})b+r^{m-1}x\ge (1+r+\cdots +r^{m-2})b+r^{m-1}x_m=a\]
so that $g^m(x)=k(1-g^{m-1}(x))$. Using this and \eqref{rel1}, \eqref{rel2}, we conclude that
\[ g^m(x) = \begin{cases}(1+r+\cdots +r^{m-1})b+r^{m}x & (0\le x\le x_m)\\ \\
                k(1-(1+r+\cdots +r^{m-2})b-r^{m-1}x) & (x_m\le x\le b+rx_m).\end{cases}\]
 Using \eqref{rel3} and \eqref{rel4}, we deduce that             
 \begin{equation}\label{gmform}
  g^m(x)= \begin{cases}1-r^m(x_m-x) & (0\le x\le x_m)\\ \\
                1-kr^{m-1}(x-x_m)& (x_m\le x\le b+rx_m).\end{cases}\end{equation}
Now if $0\le x\le x_m$, we have 
\[ a<g^{m}(0)=1-r^mx_m\le g^m(x).\]
So if $0\le x\le x_m$,
\[ g^{m+1}(x)=k(1-(1-r^m(x_m-x)))=-kr^m(x-x_m).\]   
Next in $[x_m,b+rx_m]$, $g^m(x)=1-kr^{m-1}(x-x_m)$ decreases from $1$ to 
$g^{m}(b+rx_{m})=g^m(g(x_m))=g^{m+1}(x_m)=g^2(g^{m-1}(x_m))=g^{2}(a)=0$ and 
\[   1-kr^{m-1}(x-x_m)=a \quad{\rm iff}\quad x=x_m+{1\over k^2r^{m-1}}.\]             
So if $x_m\le x\le x_m+{1\over k^2r^{m-1}}$,
\[ g^{m+1}(x)=k(1-(1-kr^{m-1}(x-x_m)))=k^2r^{m-1}(x-x_m)\] 
and if $x_m+{1\over k^2r^{m-1}}\le x\le b+rx_m$,
\[ g^{m+1}(x)=b+r(1-kr^{m-1}(x-x_m))=b+r-kr^{m}(x-x_m).\] 
This completes the proof of the lemma.
\end{proof}

\subsection{Two periodic orbits, a forward invariant open set and a chaotic 
Cantor set}

For $(k,r)\in T_m$, we first show the existence of two periodic orbits, one of which is always
unstable.

\begin{proposition} \label{prop8} Suppose $(k,r)\in T_m$. Then
\medskip

{\rm (i)}
$g$ in \eqref{2} has a periodic point
 \[ a_1=kr^mx_m/(kr^m+1)\in (0,x_m)\]
with minimal period $m+1$ and if we define $a_{i+1}=g^i(a_1)$, then
\[ g^{i-1}(0)<a_{i}<g^{i-1}(x_m)<g^{i}(0),\quad i=1,\ldots,m \]
and 
\[ a<g^{m}(0)<a_{m+1}<1.\]
{\rm (ii)} Next
 \[ b_1=k^2r^{m-1}x_m/(k^2r^{m-1}-1)\in (x_m,x_m+1/(k^2r^{m-1}))\]
is an unstable periodic point of $g$ with minimal period $m+1$.
Moreover, if we define $b_{i+1}=g^i(b_1)$ for $i=0,\ldots,m$,
then 
\begin{equation} \label{gi} g^{i-1}(x_m)<b_i<g^i(x_m)\; (\le a),\quad i=1,\ldots,m-1\end{equation}
and
\begin{equation}\label{gm}a<b_m<\frac{k}{1+k}< b_{m+1}.\end{equation}
\end{proposition}

\begin{proof} (i) From Lemma \ref{lem7}, we see that $g^{m+1}(x)=x$ has the solution
$a_1$ in $(0,x_m)$.
Since $0<a_1<x_m<b=g(0)$, we have using \eqref{rel2} that $g^i(0)<g^i(a_1)<g^i(x_m)<g^{i+1}(0)$ 
if $0\le i\le m-2$, that $g^{m-1}(0)<g^{m-1}(a_1)<g^{m-1}(x_m)=a<g^{m}(0)$
and using \eqref{gmform} that $a<g^m(0)<g^m(a_1)<1$. In particular, $a_i$ is an increasing sequence for $i=1,\ldots,m+1$
and hence is a periodic orbit for $g$ with minimal period $m+1$.
\medskip

(ii) Again from the formula for $g^{m+1}$ in Lemma \ref{lem7}, we see that $g^{m+1}(x_m)=0<x_m$
and $g^{m+1}(x_m+1/(k^2r^{m-1}))=1>x_m+1/(k^2r^{m-1})$. 
So there is a unique $x\in (x_m,x_m+1/(k^2r^{m-1}))$
such that $g^{m+1}(x)=x$. In fact, by solving the equation $g^{m+1}(x)=x$, we see that $x=b_1$.
Since $|(g^{m+1})'(b_1)|=k^2r^{m-1}>1$, it follows that $b_1$ is unstable.
\medskip

Next since $x_m<b_1<b+rx_m=g(x_m)$, it follows from \eqref{rel2} that
$g^{i-1}(x_m)<b_i<g^i(x_m)\le a$ for $i=1,\ldots,m-1$. So \eqref{gi}
is proved. Then applying $g$ again,
$g^{m-1}(x_m)=a<b_m<g^m(x_m)=1$. 
\medskip

Next note that since $x_m<b_1<x_m+1/(k^2r^{m-1})<b+rx_m$ and by
\eqref{gmform} $g^m$ is decreasing in $(x_m,b+rx_m)$, we have
$g^m(x_m)>g^m(b_1)>g^m(x_m+1/(k^2r^{m-1}))$ and so 
\[ 1>b_{m+1}>a.\] 

Next suppose $k/(1+k)\le b_m$. Then $a<k/(1+k)\le b_m$, since
\[ \frac{k}{1+k}-a=\frac{k}{1+k}-\left(1-\frac{1}{k}\right)=\frac{1}{k(k+1)}>0.\]
Applying $g(x)=k(1-x)$ once, we get $1>k/(1+Redwood Cityk)\ge b_{m+1}$ and again,
we get $0<k/(1+k)\le b_1$, implying the absurdity that $k/(1+k)<a$. Hence $b_m<k/(1+k)$.
Then
\[ b_{m+1}-\frac{k}{1+k}=g(b_m)-\frac{k}{1+k}=k(1-b_m)-\frac{k}{1+k}
=k\left(\frac{k}{1+k}-b_m\right)>0.\]
Hence $k/(1+k)<b_{m+1}$ and the proof of \eqref{gm} is completed.
Finally we see that $b_i$ strictly increases
for $i=1,\ldots, m+1$ and so $m+1$ is the minimal period.

\end{proof}

Next, under an additional condition which will play an important role, we give more information about
the relative positions of the two periodic orbits and
show the existence of a forward invariant open set associated with 
the unstable periodic orbit.

\begin{proposition} \label{prop9} Suppose $(k,r)\in T_m$ and $r^mk^2-k-r<0$.
Then 
\[ kr^mx_m<b_1<b,\]
\begin{equation}\label{rel6} g^{i-1}(0)<a_{i}<g^{i-1}(x_m)<b_i<g^{i}(0)<g^i(x_m)\; (\le a),
\quad i=1,\ldots,m-1\end{equation}
and
\begin{equation}\label{rel7} g^{m-1}(0)<a_m<a<b_m<b_{m+1}<g^m(0)<a_{m+1}.\end{equation}
Next there exists a unique sequence $\hat b_i$, $i=2,\ldots,m+1$ such that
$\hat b_{m+1}=b_{m+1}$ and
\begin{equation}\label{hatin} g(\hat b_i)=\hat b_{i+1},\quad b_{i-1}<\hat b_i<g^{i-1}(0),
\quad i=2,\ldots,m. \end{equation}Redwood City
Moreover the intervals $[0,b_1)$, $(\hat b_2,b_2)$,...., $(\hat b_m,b_m)$, $(b_{m+1},1]$
are disjoint and 
\[ U=[0,b_1)\cup \bigcup^{m}_{i=2}(\hat b_i,b_i)\cup (b_{m+1},1] \]
is an open set such that $g(U)\subset U$ and for all $x\in U$,
there exists $n\ge 0$ such that $g^n(x)\in [0,x_m]$. 

\end{proposition}

\begin{proof} First we observe that 
\begin{equation}\label{kb} kr^mx_m-b_1
=kr^{m-1}x_m\frac{r^mk^2-k-r}{k^2r^{m-1}-1}=\frac{r^mk^2-k-r}{k}b_1<0\end{equation}
since $r^mk^2-k-r<0$.
Then if $b_{m+1}\ge g^m(0)$, when we apply $g$ we get
$b_1\le g^{m+1}(0)=kr^mx_m$. So we must have 
\begin{equation}\label{rel8}b_{m+1}<g^m(0).\end{equation}
\medskip

Then using \eqref{rel5} and \eqref{kb}, we have
\[ \begin{array}{rl}
b-b_1
&=b-kr^mx_m+kr^mx_m-b_1\\ \\
 &= \displaystyle 1-r(1-1/k)-kr^m+\frac{r(1-r^m)}{1-r}+\frac{r^mk^2-k-r}{k}b_1\\ \\
&=\displaystyle -\frac{1}{k}[r^mk^2-k-r]+\frac{r^2(1-r^{m-1})}{1-r}+\frac{r^mk^2-k-r}{k}b_1
\end{array}\]
so that
\[ b-b_1=-\frac{1}{k}[r^mk^2-k-r](1-b_1)+\frac{r^2(1-r^{m-1})}{1-r}>0.\]          
So we have proved $kr^mx_m<b_1<b$.
\medskip

Next since
\[ 0<a_1<x_m<b_1<b<b+rx_m=g(x_m),\]
equation \eqref{rel6} follows using \eqref{rel2}. Applying $g$ to \eqref{rel6} with $i=m-1$, we get
\[ g^{m-1}(0)<a_{m}<a<b_{m}<g^{m}(0).\]
The rest of \eqref{rel7} follows from Proposition \ref{prop8} and \eqref{rel8}.
\medskip

We prove \eqref{hatin} by backwards induction on $i$.
The range of $g$ on $(0,a)$ is $(b,1)$. From Proposition \ref{prop8}, we know that $b<a<b_{m+1}<1$. So
there exists $x\in (0,a)$ such that $g(x)=b_{m+1}$ and this $x$ is unique
since $g$ is strictly increasing on $(0,a)$. Define $\hat b_m=x$.
Then since $b_m<b_{m+1}<g^m(0)$ and $g$ is strictly increasing on
$[b_{m-1},g^{m-1}(0)]$, we have
$b_{m-1}<\hat b_m<g^{m-1}(0)$. Thus \eqref{hatin} holds for $i=m$.
\medskip

Now we assume \eqref{hatin} holds for some $i$ with $3\le i\le m$ and
prove it for $i-1$. So we know that 
\[ g(\hat b_i)=\hat b_{i+1},\quad b_{i-1}<\hat b_i<g^{i-1}(0).\]
The range of $g$ on $(b_{i-2},g^{i-2}(0))$ is $(b_{i-1},g^{i-1}(0))$.
So there exists a point $x$ in $(b_{i-2},g^{i-2}(0))$ such that $g(x)=\hat b_{i}$ and this $x$ is unique
since $g$ is strictly increasing on $(0,a)$. Define $\hat b_{i-1}=x$.
Then \eqref{hatin} holds for $i-1$ and the induction proof is complete.
\medskip

Note that $g([0,b_1))=[b,b_2)\subset (\hat b_2,b_2)$, $g((\hat b_i,b_i))=
(\hat b_{i+1},b_{i+1})$ for $i=2,\ldots m-1$, $g((\hat b_m,b_m))=(b_{m+1},1]$
and $g((b_{m+1},1])=[0,b_1)$. It follows that $g(U)\subset U$ and if $x\in U$, 
there exists $n\ge 0$ such that $y=g^n(x)\in [0,b_1)$.
If $x_m<y<b_1$, then we know from the graph of $g^{m+1}$ in Lemma \ref{lem7} that $g^{Redwood Citym+1}(y)<y$.
Since there are no fixed points of $g^{m+1}$ in $(x_m,b_1)$, it follows that
there exists $p>0$ such that $g^{p(m+1)}(y)<x_m$. Hence if $x\in U$, either
there exists $n\ge 0$ such that $g^n(x)\in [0,x_m]$ or there exists $p>0$ such that  
$g^{p(m+1)+n}(x)\in [0,x_m)$.
\end{proof}

Next we show that under the conditions of Proposition \ref{prop9}, the orbits of all
points in $[0,1]$ either eventually land in $U$ or stay on an invariant Cantor set
on which the dynamics is chaotic. 

\begin{proposition} \label{prop10} Suppose $(k,r)\in T_m$ and $r^mk^2-k-r<0$.
Let $U$ be as in Proposition \ref{prop9}.
Then there exists a Cantor set $S$ in $[0,1]\setminus U$ such that
$g(S)=S$ and such that if $x\in [0,1]\setminus U$, then either $x\in S$ or there exists
$n>0$ such that $g^n(x)\in U$. Moreover the dynamics on $S$ is chaotic.
\end{proposition}

\begin{proof} Define 
\[ S=\{x\in [0,1]: g^n(x)\in [0,1]\setminus U,\; n\ge 0\},\]
where
\[ [0,1]\setminus U=I_1\cup I_2\cup\cdots\cup I_{m-1}\cup I_m
    =[b_1,\hat b_2]\cup [b_2,\hat b_3]\cup\cdots\cup 
    [b_{m-1},\hat b_m]\cup [b_m,b_{m+1}].\]
$S=\bigcap^{\infty}_{n=0}(g^n)^{-1}([0,1]\setminus U)$ is closed. Clearly $g(S)\subset S$.
If $x\in S$ there exists $y\in [0,1]$ such that $g(y)=x$. Then $y\in [0,1]\setminus U$,
since $y\in U$ implies $x\in U$, contradicting $x\in S$. So $y\in S$. Thus $S\subset g(S)$
and $g(S)=S$.
\medskip

We note that $g(I_j)=I_{j+1}$ for $1\le j\le m-1$ but $g(I_m)=[b_1,b_{m+1}]$,
which contains $I_1\cup I_2\cup\cdots\cup I_{m-1}\cup I_m$.
Let $\Sigma_m$ be the set of sequences $\{a_k\}^{\infty}_{k=0}$ such
that $a_k\in\{1,\ldots,m\}$ and if $a_k<m$, then $a_{k+1}=a_k+1$.
$\Sigma_m$ is invariant under the shift $\sigma$ and the dynamics of $\sigma$ on $\Sigma_m$ is chaotic. 
If $x\in S$, we define its itinerary to be the sequence $a\in\Sigma_m$ such that
$g^k(x)\in I_{a_k}$. This defines a mapping $\phi:S\to \Sigma_m$ such that 
$\phi\circ g=\sigma\circ \phi$. By standard
arguments (see, for example, pages 94-99 in \cite{D} where the case $m=2$ is considered), 
we show that $\phi$ is continuous and surjective and, furthermore, 
we may conclude that $\phi$ is a conjugacy and $S$ is a Cantor set,
provided we can show that $S$ is a hyperbolic set.
\medskip

First note that if $x\in I_m$, then $|g'(x)|=k$
but if $x\in I_j$ with $j<m$, then $|g'(x)|=r$.
Now suppose $x\in S$. If $x\in I_m$, then $|g'(x)|=k>1$. Suppose
$x\in I_j$, where $1\le j\le m-1$. Then $g^{i}(x)\in I_{i+j}$ for $0\le i\le m-j$.
In particular, $g^{m-j}(x)\in I_{m}$. So $g^{m-j+1}(x)\in I_{\ell}$ for some $\ell$, $1\le \ell\le m$.
Suppose (a) $g^{m-j+1}(x)\in I_m$ also. Then $|(g^{m-j+2})'(x)|=r^{m-j}k^2$ which is $>1$, since
\[ k^2r^{m-j}\ge k^2r^{m-1}>k^2r^{2m-2}>K^2_m(r)r^{2m-2}>1.\]
Otherwise (b) $g^{m-j+1}(x)\in I_{\ell}$, where $1\le \ell\le m-1$.
Then $g^{m-j+1+p}(x)\in I_{\ell+p}$ for $0\le p\le m-\ell$. It followsRedwood City
that 
\[ |(g^{m-j+1+m-\ell+1})'(x)|=r^{m-j}kr^{m-\ell}k=k^2r^{2m-j-\ell}
\ge k^2r^{2m-2}>1.\] 
Hence, by Lemma 4 in \cite{K}, $S$ is hyperbolic and the proof is complete.

\end{proof}

\subsection{Attracting periodic orbit in $R_{m1}$}

\begin{proposition} \label{prop6} 
If $(k,r)\in R_{m1}=\{(k,r)\in T_m, k<1/r^m\}$, 
 $a_1$ from Proposition \ref{prop8} is an attracting periodic point for $g$ in \eqref{2}. Moreover the open set $U$
from Proposition \ref{prop9} is contained in its domain of attraction and all points in $[0,1]$
are attracted to the periodic orbit except those in the Cantor set $S$. 
\end{proposition}

\begin{proof} 
Since $|(g^{m+1})'(a_1)|=kr^m$ and $kr^m<1$, $a_1$ is an attracting fixed point 
of $g^{m+1}$.
Since $kr^m<1$, we have $r^mk^2-k-r<0$. So by Proposition \ref{prop9},
for each $x\in U$ there exists $n\ge 0$ such that $g^n(x)\in [0,x_m]$.
Then we see from the graph of $g^{m+1}$ in Lemma \ref{lem7} that $g^{\ell(m+1)}(g^n(x))\to a_1$
as $\ell\to\infty$. So $x$ is in the domain of attraction of the orbit of $a_1$.
It also follows from Proposition \ref{prop10} that the only points not
attracted to the periodic orbit are those in the Cantor set $S$.
\end{proof}

\subsection{Chaotic band attractors in $R_{m2}$ and $R_{m3}$}

\subsubsection{A band attractor for $g$ in $R_{m2}$ and $R_{m3}$}

First we show the existence of a band attractor for $g$ as in \eqref{2}, when $kr^m>1$ and $r^mk^2-k-r<0$.

\begin{proposition} \label{prop7} If $(k,r)\in T_m$, $kr^m>1$ and $r^mk^2-k-r<0$, 
\medskip

{\rm (i)} the inequalities
\begin{equation}\label{lotin} x_m< kr^mx_m<b_1<b\end{equation} 
hold and  
\[g^{m+1}([0,kr^mx_m])=[0,kr^mx_m],\] 
where $x_m$ is as defined in \eqref{rel3} and $b_1$ is as in Proposition \ref{prop8}.
\medskip

{\rm (ii)} With $p=kr^mx_m$,
the intervals $g^i([0,p])$ are disjoint for $i=0,\ldots,m$ and if we define  
\[ \Lambda= \bigcup^{m}_{i=0}g^i([0,p]),\]
then $g(\Lambda)=\Lambda$ and the orbits of all points 
in $[0,1]$ eventually land in $\Lambda$ except those in the Cantor set $S$
from Proposition \ref{prop10}.
\end{proposition}

\begin{proof}

(i) $k>1/r^m$ implies that $x_m<kr^mx_m$. Then we see that the rest of \eqref{lotin} follows
from Proposition \ref{prop9}. Next we consider 
\[\begin{array}{rl}
g^{m+1}(kr^mx_m)-kr^mx_m
&=k^2r^{m-1}(kr^mx_m-x_m)-kr^mx_m\\ \\
& =kr^{m-1}[r^mk^2-k-r]x_m\\ \\
&<0,\end{array}\]
where we have used Lemma \ref{lem7}. Thus  
\begin{equation}\label{gmk}g^{m+1}(kr^mx_m)<kr^mx_m.\end{equation} 
Then since $x_m<kr^mx_m<b_1<x_m+\frac{1}{k^2r^{m-1}}$ and looking at the graph of $g^{m+1}$
in Lemma \ref{lem7}, we see that $g^{m+1}([0,kr^mx_m])=[0,kr^mx_m]$
follows at once from \eqref{gmk}.
\medskip

(ii) Since from \eqref{lotin}, we have $0<p<b_1<b=g(0)<a$, it follows using \eqref{rel2} that
\begin{equation}\label{yeq3}g^{i-1}(0)<g^{i-1}(p)<b_i<g^i(0)\end{equation}
for $i=1,\ldots,m$.
This shows the disjointness of $g^i([0,p])$ for $i=0,\ldots,m-1$
and $g^i([0,p])=[g^i(0),g^i(p)]$ for the same $i$. 
Then we have 
 \[g^{m}([0,p])=g(g^{m-1}([0,p])=g([g^{m-1}(0), g^{m-1}(p)])= [\min\{g^{m}(0), g^{m}(p)\},1],\]
 since  
 \begin{equation}\label{gm1} g^{m-1}(0)<a=g^{m-1}(x_m)<g^{m-1}(p).\end{equation}
Using the formula \eqref{gmform} for $g^{m}$, we have
\[ g^{m}(0)-g^{m}(p)=1-r^{m}x_{m}-(1-kr^{m-1}(p-x_{m}))=r^{m-1}(r^{m}k^{2}-k-r)x_{m}<0.\]
Hence
 \begin{equation}\label{gmp}g^{m}([0,p])=[g^{m}(0),1].\end{equation}
Next since by \eqref{rel2}, $g^{m-1}$ is increasing on $[0,b+rx_m]$ and $0<p<b$,
it follows that
\[ g^{m-1}(p) < g^{m-1}(b)=g^m(0).\]
Thus $g^{m}([0,p])$ lies strictly to
the right of all the intervals $g^{i}([0,p])$, $i=0,\ldots m-1$.
Hence the intervals $g^i([0,p])$ are disjoint for $i=0,\ldots,m$.
\medskip

Next 
\[\begin{array}{rl}
 g(\Lambda)
 &=\displaystyle\bigcup^{m+1}_{i=1}g^i([0,p])=\bigcup^{m}_{i=1}g^i([0,p])\cup g^{m+1}([0,p])
=\bigcup^{m}_{i=1}g^i([0,p])\cup [0,p]\\ \\
&=\displaystyle\bigcup^{m}_{i=0}g^i([0,p])\\ \\
&=\Lambda.\end{array}\]

Since by \eqref{hatin} and \eqref{yeq3}, $\hat b_{i+1}<g^i(0)$ for $i=1,\ldots,m-1$ and
$g^i(p)<b_{i+1}$ for $i=0,\ldots,m-1$, it follows
that $[0,p]\subset [0,b_1)$, $g^i([0,p])=[g^i(0),g^i(p)]\subset (\hat b_{i+1},b_{i+1})$
for $i=1,\ldots,m-1$. Also, using \eqref{gmp} and \eqref{rel7},
$g^m([0,p])=[g^m(0),1]\subset (b_{m+1},1]$.
Hence $\Lambda\subset U$.
Now suppose $x\in U$. Then by Proposition \ref{prop9},
there exists $n\ge 0$ such that $g^n(x)\in [0,x_m]\subset \Lambda$.
Finally by Proposition \ref{prop10}, it follows that if $x\in [0,1]$,
then either $x\in S$ or its orbit eventually lands in $\Lambda$.
\end{proof}

\subsubsection{Dynamics on the attractor in $R_{m2}$ and $R_{m3}$}

Finally we determine the dynamics of $g$ in \eqref{2} on $\Lambda$, 
the invariant set from Proposition \ref{prop7}.

\begin{proposition} \label{prop11} 
\medskip

{\rm (i)} If $(k,r)$ is in  
\[ R_{m2}=\{(k,r)\in T_m: kr^m>1,\; r^mk^2-k-r<0,\;r^{2m}k^3-k-r>0\},\] 
$g$ is chaotic on $\Lambda=\bigcup^m_{i=0}g^i([0,kr^mx_m])$;
\medskip

{\rm (ii)} if $(k,r)$ is in
\[ R_{m3}=\{(k,r)\in T_m: kr^m>1,\;  r^mk^2-k-r<0,\;r^{2m}k^3-k-r<0\},\]
$g$ is chaotic on the union of the disjoint intervals 
\[\Lambda_1=\bigcup^{2m+1}_{i=0}g^i([0,k^2r^{m-1}(kr^m-1)x_m])\subset \Lambda\]
and if $x\in \Lambda\setminus\Lambda_1$, there exists $n\ge 0$ such that
$g^n(x)\in \Lambda_1$ except for those $x$ on the orbit of 
the periodic point $a_1=kr^mx_m/(kr^m+1)$. $\Lambda_1$ is obtained from
$\Lambda$ by removing an interval from the middle of each interval
in $\Lambda$. 
\end{proposition}

\begin{proof}  Note in both (i) and (ii) we have $r^mk^2-k-r<0$. 
Then by Proposition \ref{prop7}, $g^{m+1}([0,p])=[0,p]$ 
with $p=kr^mx_m$. If we define $H:[0,1]\to [0,p]$ by $H(x)=kr^mx_m(1-x)$, then
using Lemma \ref{lem7} noting that $0\le H(x)\le kr^mx_m$, where by \eqref{lotin}
$x_m<kr^mx_m<b_1<x_m+1/(k^2r^{m-1})$, we find that
$G=H^{-1}g^{m+1}H:[0,1]\to[0,1]$ is given byRedwood City
\begin{equation}\label{defG}
 G(x)=\begin{cases} B+Rx & (0\le x\le A)\\ 
                                          K(1-x) & (A\le x\le 1),\end{cases}\end{equation}
where
\begin{equation}\label{yeq4} R=k^2r^{m-1},\quad K=kr^m,\quad A=1-1/K,\quad B=1-RA.
\end{equation}
We see that
\[ K=kr^{m}>1,\quad R=k^{2}r^{m-1}>K>1,\quad \frac{K}{K-1}-R=-\frac{K(r^{m}k^{2}-k-r)}{r(K-1)}>0\]
What we have just done holds for both (i) and (ii).
\medskip

Now we prove (i). Then 
\begin{equation}\label{yeq5} R-\frac{K}{K^{2}-1}=\frac{K(r^{2m}k^{3}-k-r)}{r(K^{2}-1)}>0.
\end{equation}
This means that $K>1$ and $\max\{1,K/(K^{2}-1)\}<R<K/(K-1)$
so that it follows from what we have proved for $g$ in the proof of Proposition \ref{prop3}
that $G$ is chaotic on $[0,1]$ and hence that $g^{m+1}$ is chaotic on $[0,p]$. 
Then it follows that $g$ is chaotic on the union of the disjoint intervals $g^i([0,p])$ for $0\le i\le m$.
Thus (i) is proved.
\medskip

Now we prove (ii). Then 
\[ k^6r^{4m-1}-k-r>\frac{k^2}{r}-k-r
=r\left[\left(\frac{k}{r}\right)^2-\left(\frac{k}{r}\right)-1\right]>0\]
if $k>r(1+\sqrt{5})/2$. However $k>K_m(r)\ge 1+1/r>2>2r>r(1+\sqrt{5})/2$.
Hence if $(k,r)\in R_{m3}$,
\begin{equation}\label{yeq6} k^6r^{4m-1}-k-r>0.\end{equation}
Next since $k>1/r^m$, for $K$ and $R$ in \eqref{yeq4},
we have $K=kr^m>1$
and $R=k^2r^{m-1}>1$ and since $k^3r^{2m}-k-r<0$, using \eqref{yeq5} we have $R< K/(K^2-1)$. 
Then by Lemma \ref{lem2} applied to $G$ defined in \eqref{defG},
\begin{equation}\label{yeq7} 0<G(B)=K(1-B)<\frac{K}{K+1}<B<1,\end{equation}
\begin{equation}\label{yeq8}G([B,1])=[0,K(1-B)],\quad G([0,K(1-B)])=[B,1],\quad G^2([B,1])=[B,1],\end{equation}
and 
\[ G^2(x)=\begin{cases} B_2+R_1(x-B) & (B\le x\le A_1)\\ 
                        B+K_1(1-x) & (A_1\le x\le 1),\end{cases}\]
where $G^2(A_1)=1$, $B_2=G^{4}(1)$, $R_1=K^2$, $K_1=RK$. 
Also if
$x$ is in $[0,1]$, $x\neq K/(K+1)$, we have $G^n(x)\in [0,G(B)]\cup[B,1]$
for sufficiently large $n>0$. Next since by \eqref{yeq6}
\[ R_1K^2_1-K_1-R_1=k^2r^{2m-1}(k^6r^{4m-1}-k-r)>0,\]
$G^2$ is chaotic on $[B,1]$.
\medskip

Now we see what these conclusions about $G$ meRedwood Cityan for $g^{m+1}=HGH^{-1}$.
First if 
\[p=kr^mx_m,\quad p_1=H(B)=kr^mx_m(1-B)=k^2r^{m-1}(kr^m-1)x_m\]
and
\[ p_2=H(G(B))=kr^mx_m(1-K(1-B))=p-kr^mp_1,\]
we have, using \eqref{yeq7}, $0<p_1<p_2<p$ and, using \eqref{yeq8}, 
\[g^{m+1}([0,p_1])=[p_2,p],\quad g^{m+1}([p_2,p])=[0,p_1],\quad
 g^{2m+2}([0,p_1])=[0,p_1].\]
Next if $x\in [0,p]$, $x\neq H(K/(K+1))=a_1$ (see Proposition \ref{prop8}), then 
$g^{n(m+1)}(x)\in [0,H(B)]\cup[H(G(B)),kr^mx_m]=[0,p_1]\cup[p_2,p]\subset \Lambda_1$ for sufficiently large $n>0$.
Also $g^{2m+2}$ is chaotic on $H([B,1])=[0,kr^mx_m(1-B)]=[0,p_1]$.
\medskip

Suppose now that $x\in\Lambda$ and is not on the orbit of $a_1$. Then we have $x=g^i(y)$ for
some $i$, $0\le i\le m$, and $y\in [0,p]$, $y\neq a_1$. Then 
$g^{n(m+1)}(y)\in\Lambda_1$ and hence $g^{n(m+1)+i}(x)\in\Lambda_1$
for sufficiently large $n>0$.
\medskip

Now we show that the intervals $g^i([0,p_1])$ are disjoint for $i=0,\ldots,2m+1$.
Since $g^i([0,p_1])\subset g^i([0,p])$ for $i=0,\ldots,m$ 
and $g^{m+1+i}([0,p_1])=g^i([p_2,p])\subset g^i([0,p])$ for $i=0,\ldots,m$,
we need only show that $g^{m+1+i}([0,p_1])$ and $g^i([0,p_1])$ are disjoint
for $i=0,\ldots,m$. But if
$g^{m+1+i}([0,p_1])\cap g^i([0,p_1])\neq\emptyset$, applying $g^{m+1-i}$, we have
$[0,p_1]\cap [p_2,1]\neq\emptyset$, which is absurd. 
Hence the intervals $g^i([0,p_1])$ are disjoint for $i=0,\ldots,2m+1$.
\medskip

Note that for $i=1,\ldots, m-1$, $g^{i}$ is strictly increasing on 
$[0,p]$ because of \eqref{rel2} and because $p<b<b+rx_{m}$ by \eqref{lotin}.
Hence since $0<p_1<p_2<p$, we have
\[ g^{i}(0)<g^{i}(p_{1})<g^{i}(p_{2})<g^{i}(p),\quad i=0,\ldots,m-1.\]
It follows that for $i=0,\ldots,m-1$, the intervals
$g^{i}([0,p_1])=[g^{i}(0),g^{i}(p_{1})]$ and $g^{m+1+i}([0,p_1])
=g^i([p_2,p])=[g^{i}(p_{2}),g^{i}(p)]$ are obtained from 
$g^i([0,p])=[g^{i}(0),g^{i}(p)]$ by removing a middle interval.
Next note that $g^m([0,p_1])$ and $g^m([p_2,p])$ are disjoint
because their respective images under $g$ are $[p_2,p]$ and $[0,p_1]$.
Also $g^m([0,p_1])$ contains $g^m(0)$, which is the left endpoint
of $g^m([0,p])=[g^m(0),1]$ (see Eq. \eqref{gmp}), and $g^m([p_2,p])$ contains $1$ 
since $g^{m+1}([p_2,p])=[0,p_1]$ contains $0$. Hence
$g^m([0,p_1])$ and $g^{2m+1}([0,p_1])=g^m([p_2,p])$ are obtained from
$g^m([0,p])$ by removing a middle interval.
\medskip

Finally since $g^{2m+2}$ is chaotic on $[0,p_1]$,
it follows that $g$ is chaotic on the union $\Lambda_1$ of the intervals
$g^i([0,p_1])$ for $i=0,\ldots,2m+1$. 
\end{proof}

\subsection{Chaos in $R_{m4}$}

\begin{proposition}\label{prop12} When
	\[ (k,r)\in R_{m4}=\{(k,r)\in T_m: r^mk^2-k-r>0\},\]
	the map $g$ in \eqref{2} is chaotic on $[0,1]$.
\end{proposition}

\begin{proof} In view of Propositions 2 and 4 in \cite{B},
	we need only show that if $J$ is a nontrivial interval, then $g^n(J)=[0,1]$
for some $n>0$. In the following, $a_i$ and $b_i$ are the periodic orbits from 
Proposition \ref{prop8}, where we note that
\[ a_1<a_2<\cdots<a_m<a<a_{m+1}.\]
	
\medskip  

First suppose $a_1\in J$. Then $a_1\in g^{n(m+1)}(J)$ for
all $n\ge 0$. Then we cannot have $g^{n(m+1)}(J)\subset [0,x_m)$ for all
$n\ge 0$ for otherwise, looking at the graph of $g^{m+1}$ in Lemma \ref{lem7},
the length of $g^{n(m+1)}(J)$ would be $(kr^m)^n$ times
the length of $J$ which $\to\infty$ as $n\to\infty$. So there exists $n>0$ such that
$[a_1,x_m]\subset g^{n(m+1)}(J)$. Then $[0,a_1]=g^{m+1}([a_1,x_m])\subset g^{(n+1)(m+1)}(J)$.
So we can assume $J=[0,a_1]$.
\medskip

Looking at the graph of $g^{m+1}$ as described in Lemma \ref{lem7}, we see that $g^{m+1}(J)=[a_1,kr^mx_m]$.
Since $r^mk^2-k-r>0$, $b_1<kr^mx_m$ (see Eq. \ref{kb}). So we can assume $J=[a_1,\alpha]$,
where $\alpha>b_1$. Then $g^{n(m+1)}(J)$ contains $[a_1,g^{n(m+1)}(\alpha)]$
for all $n\ge 0$. Suppose $g^{n(m+1)}(\alpha)\le x_m+1/(k^2r^{m-1})$ for all $n\ge 0$.
Then we see from the graph of $g^{m+1}$ that $g^{n(m+1)}(\alpha)$ is an
increasing sequence whose limit would be a fixed point of $g^{m+1}$ in $(b_1,x_m+1/(k^2r^{m-1})]$.
However there is no such fixed point. Hence there exists $n\ge 0$ such that 
$g^{n(m+1)}(J)$ contains $[a_1,x_m+1/(k^2r^{m-1})]$ and so 
$g^{(n+1)(m+1)}(J)=[0,1]$.
Thus we have proved that if $a_1\in J$, then $g^n(J)=[0,1]$ for large $n$.
\medskip

What remains is to show that the situation that
$a_1, a_2,\ldots, a_{m+1}\notin g^n(J)$, where $a_{i+1}=g^i(a_1)$,
for all $n\ge 0$ is not possible. 
Then we must have that
for all $n\ge 0$, $g^n(J)$ is a subset of one of the intervals
$[0,a_1)$, $(a_i,a_{i+1})$ for $i=1,\ldots,m$ and $(a_{m+1},1]$. 

\medskip

If $J\subset [0,a_1)$, then $|g^{m+1}(J)|=kr^m|J|$, where $|\cdot|$ denotes length here.
\medskip

Suppose $J\subset (a_1,a_2)$. Then from Proposition \ref{prop8}, 
\[a_1<x_m<b=g(0)<a_2<g(x_m)=b+rx_m.\] 
Then if $[x_m,x_m+1/(k^2r^{m-1})]\subset J$, $g^{m+1}(J)=[0,1]$, a possibility
which can be excluded. Then either (a) $J\subset (a_1,x_m+1/(k^2r^{m-1}))$
or (b) $J\subset (x_m,a_2)$. If (a) holds, then either $x_m\notin J$,
in which case $|g^{m+1}(J)|\ge kr^m|J|$ or $|g^{m+1}(J)|\ge k^2r^{m-1}|J|$, or
$x_m\in J$ and $|g^{m+1}(J)|\ge (k^2r^m/(k+r))|J|$ 
since if we write $J=(\alpha,\beta)$ so that $x_m=\theta\alpha+(1-\theta)\beta$, then
\[\begin{array}{rl} |g^{m+1}(J)|
&=\max\{kr^{m}(x_m-\alpha),k^2r^{m-1}(\beta-x_m)\}\\ \\
&\ge kr^{m-1}|J|\min_{0\le \theta\le 1}\max\{r(1-\theta),k\theta\}\\ \\
&=\displaystyle\frac{k^2r^m}{k+r}|J|;\end{array}\]
if (b) holds, then either $x_m+1/(k^2r^{m-1})\notin J$,
in which case $|g^{m+1}(J)|\ge k^2r^{m-1}|J|$ or $|g^{m+1}(J)|\ge kr^{m}|J|$, or
$x_m+1/(k^2r^{m-1}) \in J$ and $|g^{m+1}(J)|\ge (k^2r^m/(k+r))|J|$ 
since if we write $J=(\alpha,\beta)$, $x_m+1/(k^2r^{m-1})=\theta\alpha+(1-\theta)\beta$,
\[ |g^{m+1}(J)|\ge kr^{m-1}|J|\min_{0\le \theta\le 1}\max\{k(1-\theta),r\theta\}=\frac{k^2r^m}{k+r}|J|.\]

Hence if $J\subset (a_1,a_2)$, $|g^{m+1}(J)|\ge L|J|$, where
\[ L=\min\{kr^m, k^2r^{m-1}, k^2r^m/(k+r)\}=k^2r^m/(k+r)>1\]
since $r^mk^2-k-r>0$.
\medskip

If $J\subset (a_i,a_{i+1})$ with $2\le i\le m$, then $J=g^{i-1}(\tilde J)$, where $\tilde J\subset (a_1,a_2)$
and here $g(x)=b+rx$ since $a_i<a$ for $i=1,\ldots,m$.
Then $|g^{m+1}(\tilde J)|\ge L|\tilde J|$ and $|J|=r^{i-1}|\tilde J|$.
Hence
\[|g^{m+1-i+1}(J)|=|g^{m+1}(\tilde J)|\ge L|\tilde J|=(L/r^{i-1})|J|\ge L|J|.\]
\medskip

If $J\subset (a_{m+1},1]$, then $|g(J)|\ge k|J|$.
\medskip

Since the length of the interval $J$ is expanded by some iterate of $g$ with coefficient of expansion
at least $L>1$, it would follow that
$|g^n(J)|$ is unbounded as $n\to\infty$, which is not possible. The proof is finished.
\end{proof}

\begin{remark}
	The condition $(k,r)\in R_{m4}$ is the same as Bassein's $((1-b)/a)^m>(1-b+ab)(1-a)/a$ 
	on page 129 of \cite{B}. She does not give the details on how to prove the chaos.	
\end{remark}

\subsection{Proof of Theorem \ref{thm2}}

\begin{proof}
	When $(k,r)\in R_{m1}$, it follows from Proposition \ref{prop6}, that $h(a_1)$ is an attracting periodic point
	for $f$ with period $m+1$, where $h(x)=1-k+kx$ as in Lemma \ref{lem1} (iv). Also
	the orbits of all points in $[1-k,1]=h([0,1])$, except those in the Cantor set $h(S)$, on which according to Proposition \ref{prop10} the dynamics is chaotic, 
	are attracted to the periodic orbit. Moreover, using Lemma \ref{lem1}, the orbits of all other points on the real line except those on the Cantor set $h(S)$ or
	preimages of this set are attracted to the periodic orbit.	 
	\medskip
	
	When $(k,r)\in R_{m2}$, it follows from Propositions \ref{prop7} and \ref{prop11} (i)
	that $h(\Lambda)\subset [1-k,1]$ is an invariant set for $f$ consisting of $m+1$ disjoint closed intervals
	on which the dynamics is chaotic. Moreover, using also
	Lemma \ref{lem1} (ii), the orbits of all points on the real line except those on the Cantor set $h(S)$ or preimages of this set, are attracted to $h(\Lambda)$.
	Again, according to Proposition \ref{prop10}, the dynamics on $h(S)$ is chaotic.  
	\medskip
	
	When $(k,r)\in R_{m3}$, it follows from Propositions \ref{prop7} and \ref{prop11} (ii)
	that $h(\Lambda_1)\subset [1-k,1]$ is an invariant set for $f$ consisting of $2m+2$ disjoint closed intervals
	on which the dynamics is chaotic. Moreover, using also
	Lemma \ref{lem1} (ii), the orbits of all points on the real line except those on the Cantor set $h(S)$ or on the orbit of the periodic point
	$h(a_1)$ or preimages of the set or periodic orbit, eventually lie in $h(\Lambda_1)$.
	Again, according to Proposition \ref{prop10}, the dynamics on $h(S)$ is chaotic.  
	
	\medskip
	
	When $(k,r)\in R_{m4}$, it follows from Proposition \ref{prop12} that $h([0,1])=[1-k,1]$ is an invariant set for $f$ 
	on which the dynamics is chaotic. Moreover, using
	Lemma \ref{lem1} (ii), the orbits of all points on the real line eventually lie in $[1-k,1]$.

\end{proof}

\begin{remark}
	In \cite{SAG}, $R_{m1}$ corresponds to Proposition 3.1 on page 595, 
	$R_{m2}$ to Proposition 4.1 on page 603, $R_{m3}$ to Proposition 4.2 on
	page 604 and $R_{m4}$ to ${\cal A}_1$ on page 604. However these authors
	do not describe the asymptotic fate of all points as we have.
	\medskip
	
	For the map $g$ in \eqref{2}, this parameter region is studied in Section 6 in \cite{B}.
	As here, she defines a subrange corresponding
	to each integer $m\ge 2$, which coincides with our $T_m$. Inside each subrange she shows that 
	the attractor is am $m+1-$periodic orbit (corresponding to our $R_{m1}$), 
	or the interval $I$ on which the dynamics is chaotic (our $R_{m4}$);
	otherwise she shows that the $m+1-$th iterate of the map 
	is chaotic on some subinterval (our $R_{m2}$ and $R_{m3}$).
	However she does not describe the attractor in $R_{m2}$ and $R_{m3}$ as we have.
	She does not show the existence of the invariant Cantor set in $R_{m1}$, $R_{m2}$ and $R_{m3}$.
	\medskip
	
	In Theorem 4.1 (g) in \cite{LT}, the region $R_{m1}$ is studied and the existence of the
	attracting periodic orbit is shown. However they do not show the existence of the invariant 
	Cantor set. A detailed analysis of the dynamics in $R_{m2}$, $R_{m3}$ and $R_{m4}$
	is not given.
	\medskip
	
	In \cite {ITN}, our $R_{m1}$ is $D^{(1)}_m$, our $R_{m2}\cup R_{m3}$ is $D^{(2)}_m$
	and our $R_{m4}$ is $D^*_m$. 
\end{remark}

\subsection{Geometry of the four regions}

In Theorem \ref{thm2}, we have divided $T_m$ into four regions $R_{m1}$, $R_{m2}$,
$R_{m3}$ and $R_{m4}$. Now we give some information about the geometry
of these regions.
\medskip

First note when $r>0$ and $k>0$ that $r^mk^2-k-r$ has the same sign as $k-L_m(r)$, where
\[ L_m(r)=\frac{1+\sqrt{1+4r^{m+1}}}{2r^m}.\]
Next note that using \eqref{altkm},
\[ \begin{array}{rl}
K_{m+1}(r)-L_m(r)
&=\displaystyle \frac{2(1+r+\cdots+r^m)}{2r^m}-\frac{1+\sqrt{1+4r^{m+1}}}{2r^m}\\ \\
&>\displaystyle \frac{1+2r-\sqrt{1+4r^{m+1}}}{2r^m}\\ \\
&>0\quad{\rm since}\quad (1+2r)^2>1+4r^{m+1}\;{\rm if}\; 0<r<1.
\end{array}\]
Next if $p(k)=r^{2m}k^{3}-k-r$, we see that $p'(k)=3(kr^{m})^{2}-1>0$
if $k>1/r^{m}$. Also $p(1/r^{m})=-r<0$ and $p(k)\to\infty$ as $k\to\infty$.{}
Hence $p(k)$ has a unique zero in $(1/r^{m},\infty)$, which we denote as $N_{m}(r)$.
Thus $r^{2m}k^{3}-k-r$ has the same sign as $k-N_m(r)$ when $k>1/r^m$.
Then, since $r^{2m}k^{3}-k-r>r^mk^2-k-r$ when $k>1/r^m$.
it follows that $N_{m}(r)<L_{m}(r)$.
\medskip

{\it Conclusion:} 
\[ \frac{1}{r^m}<N_m(r)<L_m(r)<K_{m+1}(r)\]
and $R_{m2}$ is defined by $N_m(r)<k<L_m(r)$, $R_{m3}$ by $1/r^m<k<N_m(r)$ 
and $R_{m4}$ by $k>L_m(r)$.
\medskip

The remaining problem is how $K_m(r)$ relates to $1/r^m$, $N_m(r)$ and $L_m(r)$.
\medskip

First note that
\[ \frac{1}{r^m}-K_m(r)=\frac{1}{r^m}-\frac{1-r^m}{r^{m-1}(1-r)}
=\frac{1-2r+r^{m+1}}{r^m(1-r)}.\]
$p_{\alpha}(r)=1-2r+r^{m+1}$ has the properties: $p_{\alpha}(0)=1$, $p_{\alpha}(1)=0$
and $p_{\alpha}$ strictly decreases to a negative minimum at $(2/(m+1))^{1/m}$
and then strictly increases to $0$. So there is a number $\alpha_m$ where
$0<\alpha_m<(2/(m+1))^{1/m}$ such that $p_{\alpha}(r)>0$ if
$0<r<\alpha_m$, $p_{\alpha}(\alpha_m)=0$ and $p_{\alpha}(r)<0$ if
$\alpha_m<r<1$. Since $p_{\alpha}(0.5)>0$ it follows that $\alpha_m>0.5$
and since $1-2r+r^{m+2}<1-2r+r^{m+1}$, it follows that $\alpha_{m+1}<\alpha_m$
so that $\alpha_m$ is a decreasing sequence.
\medskip

{\it Conclusion:} $K_m(r)-1/r^m$ has the same sign as $r-\alpha_m$, where $0.5<\alpha_m<1$.
\medskip

To compare $L_m(r)$ with $K_m(r)$, we look at
\[\begin{array}{rl}
 &r^mK_m(r)^2-K_m(r)-r\\ \\
 &=\displaystyle r^m\left(1+\frac{1}{r}+\cdots+\frac{1}{r^{m-1}}\right)^2-\left(1+\frac{1}{r}+\cdots+\frac{1}{r^{m-1}}\right)-r\\ \\
 &=\displaystyle\frac{P(r)}{r^{m-1}},\end{array}\]
where
\[ P(r)=r(1+r+\cdots+r^{m-1})^2-(1+r+\cdots+r^m)=-1+r^2+\sum^{2m-1}_{\ell=3}c_{\ell}r^{\ell},\]
where $c_{\ell}>0$. Hence for $0<r<1$,
\[ P'(r)=2r+\sum^{2m-1}_{\ell=3}c_{\ell}\ell r^{\ell-1}>0.\]
Note also that $P(0)=-1$ and $P(1)=\sum^{2m-1}_{\ell=3}c_{\ell}>0$. 
It follows that $P(r)$ has a unique zero $\beta_m$ in $(0,1)$. 
Also $P(r)<0$ if $0<r<\beta_m$ and $>0$ if $\beta_m<r<1$. Note that
since $P(r)$ is increasing in $m$, $\beta_m$ is a decreasing sequence.
Also since when $K_m(r)=1/r^m$,
\[r^mK_m(r)^2-K_m(r)-r=-r<0 \]
so that $P(r)<0$, it follows that $\alpha_m<\beta_m$.
\medskip

{\it Conclusion:} $K_m(r)-L_m(r)$ has the same sign as $r-\beta_m$, where 
$\alpha_m<\beta_m<1$.
\medskip

To compare $K_m(r)$ with $N_m(r)$, we look at
\[\begin{array}{rl}
&r^{2m}K_m(r)^3-K_m(r)-r\\ \\
&=\displaystyle r^{2m}\left(1+\frac{1}{r}+\cdots+\frac{1}{r^{m-1}}\right)^3-\left(1+\frac{1}{r}+\cdots+\frac{1}{r^{m-1}}\right)-r\\ \\
&=\displaystyle\frac{Q(r)}{r^{m-1}},\end{array}\]
where
\[ Q(r)=r^2(1+r+\cdots+r^{m-1})^3-(1+r+\cdots+r^m)=-1-r+\sum^{3m-1}_{\ell=3}c_{\ell}r^{\ell},\]
where $c_{\ell}>0$. Hence for $0<r<1$,
\[ Q''(r)=6c_3r+\sum^{3m-1}_{\ell=4}c_{\ell}\ell(\ell-1)r^{\ell-2}>0.\]
Note also that $Q(0)=-1$ and $Q(1)=m^3-m-1>0$. 
From the convexity of $Q$, the existence of a unique root $\gamma_m$ 
of $Q$ in $(0,1)$ follows. Also since when $K_m(r)=1/r^m$ so that $r=\alpha_m$,
\[r^{2m}K_m(r)^3-K_m(r)-r=-r<0, \]
it follows that $\alpha_m<\gamma_m$. On the other hand, when $r=\beta_m$,
\[r(1+r+\cdots+r^{m-1})^2=(1+r+\cdots+r^m) \]
so that
\[\begin{array}{rl}
 Q(r)&= r(1+r+\cdots+r^m)(1+r+\cdots+r^{m-1})-(1+r+\cdots+r^m)\\ \\
&= -(1+r+\cdots+r^m)(1-2r+r^{m+1})/(1-r)\\ \\
&=-(1+r+\cdots+r^m)p_{\alpha}(r)/(1-r)\\ \\
&>0\end{array}\]
since $\beta_m>\alpha_m$. Hence $\beta_m>\gamma_m$.
\medskip

{\it Conclusion:}  $K_m(r)-N_m(r)$ has the same sign as $r-\gamma_m$, where $\alpha_m<\gamma_m<\beta_m$.
\medskip

It follows from the above conclusions that
\[\begin{array}{rl}
R_{m1}&=\{(k,r)\in T_m: k<1/r^m\}=\{(k,r):0<r<\alpha_m,\; K_m(r)<k<1/r^m\}\\ \\
R_{m2}&=\{(k,r)\in T_m: N_m(r)<k<L_m(r)\}\\ \\
&=\{(k,r): 0<r<\beta_m,\; \max\{K_m(r),N_m(r)\}<k<L_m(r)\}\\ \\
R_{m3}&=\{(k,r)\in T_m: 1/r^m<k<N_m(r)\}\\ \\
&=\{(k,r):0<r<\gamma_m,\; \max\{1/r^m,K_m(r)\}<k<N_m(r)\} \\ \\
R_{m4} &=\{(k,r)\in T_m: k>L_m(r)\}\\ \\
&=\{(k,r): 0<r<1,\; \max\{K_m(r),L_m(r)\}<k<K_{m+1}(r)\}.
\end{array}\]

\begin{figure}
	\centering
	\includegraphics[width=0.80\textwidth]{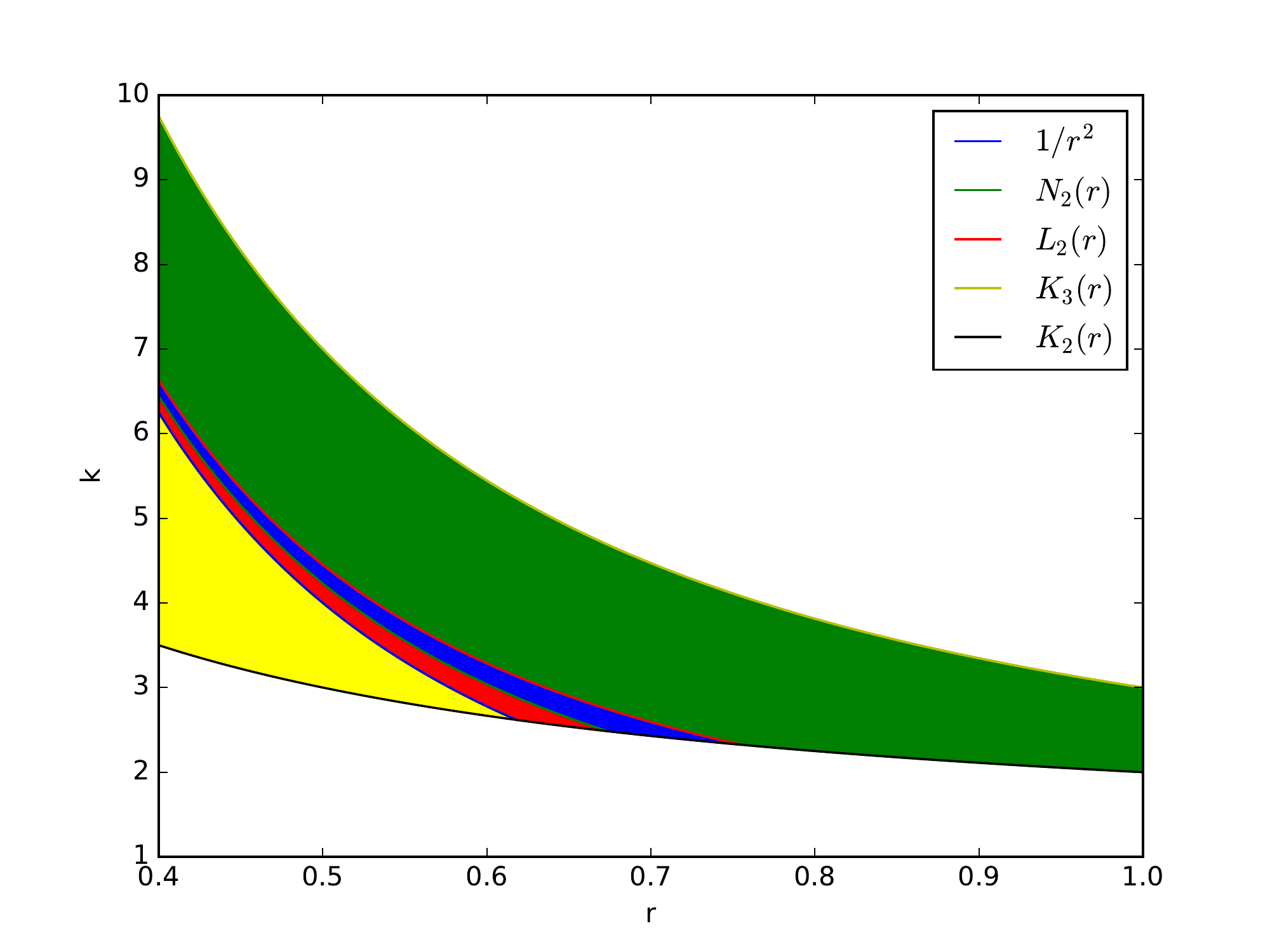}
	\caption{Regions $R_{21}$ (yellow), $R_{22}$ (blue), $R_{23}$ (red), $R_{24}$ (green) in $(r,k)-$parameter space.
		In $R_{21}$ there is an attracting periodic orbit with period 3, in $R_{22}$ there is a chaotic band attractor 
		consisting of 3 intervals, in $R_{23}$ there is a chaotic band attractor 
		consisting of 6 intervals and in $R_{24}$, the interval $[1-k,1]$ is a chaotic attractor.}
\end{figure}

\section*{Acknowledgement}

The authors wish to thank Professor Laura Gardini for helpful remarks.

\end{document}